\documentclass[preprint,12pt]{elsarticle}

\usepackage{mathtools}	

\usepackage{amssymb}
\usepackage{amsthm}
\newtheorem{thm}{Theorem}

\newtheorem{rem}{Remark}
\theoremstyle{definition}

\begin{document}

\begin{frontmatter}



\title{\Large Dynamics of interval maps generated by erasing substitutions \tnoteref{l1}}

\author{A. Della Corte\fnref{adc}}
\ead{alessandro.dellacorte@unicam.it}
\fntext[adc]{Mathematics Division, School of Sciences and Technology, University of Camerino (Italy).
Address: via Madonna delle Carceri 9A, Camerino (MC), Italy. ORCID: 0000-0002-1782-0270 (corresponding author).}

\author{M. Farotti\fnref{mf}}
\ead{marco.farotti@studenti.unicam.it.}
\fntext[mf]{Mathematics Division, School of Sciences and Technology, University of Camerino (Italy).
Address: via Madonna delle Carceri 9A - Camerino (MC), Italy.}


\begin{abstract}
\noindent We study the discrete dynamics of interval maps generated by the action of erasing block substitutions on the binary expansion. After establishing some general properties of these maps, we categorize erasing block substitutions in a hierarchy of classes displaying progressively stronger erasing character. We investigate how this affects the dynamics of the corresponding interval maps, showing that the richest dynamical behavior (Devaney and Li-Yorke chaos, infinite topological entropy) is achieved at a precise step in this hierarchy, which we name completely erasing substitutions. 
\end{abstract}




\end{frontmatter}


\newtheorem{defn}{Definition}[section]
\newtheorem{corol}{Corollary}
\newtheorem{lem}{Lemma}[section]
\section{Introduction}
\label{intr}

\noindent The study of dynamical systems has greatly benefited from the symbolic approach, and in particular from substitutive dynamics. Within this framework, several ideas have found their neat, ideal formulation, generating techniques and results that proved fruitful, for instance, in ergodic theory, chaos theory, number theory  and crystallography (a standard general reference on substitutive dynamical systems is \cite{fogg2002substitutions}). While substitutions and block substitutions are generally assumed to map symbols or words to nonempty words, recently some attention has been devoted to the properties of \textit{erasing} substitutions, which include the empty word in their  range \cite{durand2010cobham,durand2009syndeticity,reidenbach2011restricted}.
This interest is understandable, coming from a natural extension of the somewhat narrow original concept of substitution, and it will probably become wider, since a variety of real-world processes that can be potentially formalized in a symbolic dynamical context easily include erasing phenomena - think about DNA transcription and coding, information transmission with errors, model reduction for physical systems in which some states are intrinsically negligible.

\noindent In this paper, we focus on the dynamical properties of maps defined by the action of erasing block substitutions on the binary expansion of reals in the unit interval (a model case of this type is studied in detail in \cite{dellacorte2021simplest}). The systematic use of base 2 was chosen for simplicity. Working in the real context involves some technicalities, mainly due to the ambiguity of the representation of dyadic rationals and to the identification of the finite words $w0^m$ ($m\in\mathbb{N}_0$) with the infinite word $w0^\infty$. However, once formalized appropriately, the maps generated in this way look like natural hunting ground for interesting dynamical phenomena, since they lie at the current boundary of the fast-expanding domain of topological dynamics. Indeed, they are typically Baire-1, non-Darboux functions, and therefore they represent particularly simply-defined examples from classes of objects for which the study of topological dynamical properties has begun in quite a recent past (for instance, \cite{korczak2015} mostly focuses on Darboux, Baire-1 maps, while \cite{natkaniec2010pawlak,loranty2019almost} consider almost continuous functions and \cite{steele2019,steele2017,steele2018} study topologically typical Baire-1 maps). 

\noindent Our main purpose was to start the search for connections between algorithmic properties of erasing substitutions and  dynamical properties of the corresponding interval maps. Hopefully, understanding these connections may help erasing substitutions to play a role, in the investigation of highly discontinuous dynamical systems,  similar to the fruitful role played by non-erasing substitutions in the classical continuous (or piecewise continuous) case. 

\noindent The paper is organized as follows: definitions and preliminary results are provided in Section \ref{prelim}, where, in particular, erasing substitutions are categorized in a hierarchy of classes displaying progressively stronger erasing character. In Section \ref{generall} we study some general analytical and topological properties shared by the maps generated by the members of all erasing classes if they verify some form of surjectivity (named here \textit{optimality condition}). Although not particularly difficult, these results require some care (and space) to cover all the cases generated by the ambiguity of binary representation or reals. In Section \ref{dynamicss} we address the more challenging task of establishing links between the erasing class of a substitution and the dynamical behavior of the corresponding interval map; the main results are summarized in the table at the end of the section. Finally, in Section  \ref{furtherr} some possible directions for further investigation are highlighted.

\section{Preliminaries}
\label{prelim}
\noindent We set $\mathbb{I}=[0,1]$ and let $\mathbb{Q}_2=\mathbb{I}\cap \left\lbrace\frac{n}{2^{k}}:n,k\in\mathbb{N}_0\right\rbrace$ denote the dyadic rationals in $\mathbb{I}$. We set $\mathcal{Q}_2=\mathbb{Q}_2\setminus\{0,1\}$ and $\mathcal{Q}_2^0=\mathcal{Q}_2\cup\{0\}$. We indicate by $\left\lbrace 0,1\right\rbrace^*$ the set of all finite binary words and by $\left\lbrace 0,1\right\rbrace^{\infty}$ the set $\left\lbrace 0,1\right\rbrace^*\cup \left\lbrace 0,1\right\rbrace^{\omega}$ of all finite or infinite binary words. We let $\epsilon$ denote the empty word and we set $\left\lbrace 0,1\right\rbrace^+=\left\lbrace 0,1\right\rbrace^*\setminus \left\lbrace\epsilon\right\rbrace$. For $n\in\mathbb{N}$, we set 
$$\{0,1\}^{\le n}=\bigcup_{\mathbb{N}\ni i\le n}  \{0,1\}^i\ \ \ \text{,}\ \ \ \{0,1\}^{\ge n}=\left(\bigcup_{\mathbb{N}\ni i\ge n} \{0,1\}^i\right)\cup \{0,1\}^\omega$$ 

\noindent For $n\in\mathbb{N}$ and for $u\in \left\lbrace0,1\right\rbrace^{\ge n}$, we let $u_n$ denote the $n$-th digit of the binary word $u$ and for $u\in\{0,1\}^n$ we indicate by $\left|u\right|$ the length of the word $u$, i.e. the non-negative integer $n$, while we set $|u|=\infty$ if $u\in\{0,1\}^\omega$. For $v\in\{0,1\}^n$ and $m\in\mathbb{N}$ we let $v^m$ indicate the $m$-th power of the word $v$, that is the word $u\in\{0,1\}^{mn}$ such that $u_{hn+r}=v_r$ for every positive integer $r\le n$ and every non-negative integer $h< m$. We set $v^0=\epsilon$. We will use natural numbers as subscripts (superscripts) in round parentheses when they do not indicate digit number (word power). 

\noindent We will write concatenation of words using a multiplicative notation. Therefore, If $u\in\left\lbrace0,1\right\rbrace^n$ and $v\in \left\lbrace0,1\right\rbrace^\infty$, we write $uv$ for the word $u_1u_2\dots u_nv_1v_2\dots$ . Notice that, with concatenation as internal operation, the sets $\{0,1\}^*$ and $\{0,1\}^+$ have respectively the structure of a free monoid and a free semigroup. Therefore, exploiting associativity, if $\left(u^{(n)}\right)_{n=1,\dots, N}$ is a finite sequence of finite binary words, we let 
$\prod_{i=1}^N u^{(i)}=u^{(1)}u^{(2)}\dots u^{(N)}$
denote their concatenation. Moreover, if $\left(u^{(n)}\right)_{n\in\mathbb{N}}$ is an infinite sequence of finite binary words, we let 
$\prod_{i=1}^\infty u^{(i)}=u^{(1)}u^{(2)}\dots$ denote their infinite concatenation. We let $u^{\infty}$ denote the infinite concatenation $uuu\dots$ of $u$ with itself. 

\noindent We say that $v\in\{0,1\}^\infty$ is a \textit{prefix} of a (possibly infinite) word $w=a_1a_2\dots$ if  either $v=w \in \{0,1\}^\omega$ or $w \in \{0,1\}^{\ge k}$ and there is a positive integer $k$ such that $v=a_1\cdots a_k$. We say that $v\in\{0,1\}^*$ is a \textit{suffix} of a finite word $w=a_1\dots a_n$ if there is a positive integer $k< n$ such that $v=a_{n-k}a_{n-k+1}\dots a_n$.

\begin{defn}
For $k\in\mathbb{N}$, we say that $w\in\{0,1\}^\infty$ has $u\in\{0,1\}^\infty$ as a $k$-factor if there is a non-negative integer $n$ such that $w$ has a prefix of the form $pu$, with $|p|=nk$. Notice that $|u|$ can also be $\infty$.
\label{kfactor}
\end{defn}

\noindent If $w\in\{0,1\}^\infty$, we let $0.w$ denote the real number $\sum_{i=1}^{|w|} \frac{w_i}{2^i}$. We define \textit{binary expansion} of $x\in\mathbb{I}$ any word $u\in\{0,1\}^\infty$ such that $x=0.u$, so that $\mathcal{Q}^0_2$ is precisely the subset of $\mathbb{I}$ whose points do not have a unique binary expansion. If $x\in \mathbb{I}\setminus\{0\}$, we let $\widetilde{x}\in \left\lbrace0,1\right\rbrace^\omega$ denote the unique infinite binary expansion of $x$ not ending with $0^\infty$, so that $x=0.\widetilde{x}$.  Notice that $\widetilde{x}$ is the unique binary expansion of $x$ if and only if $x\notin\mathcal{Q}_2^0$. For $w\in\{0,1\}^*$, we indicate by $[w]$ the cylinder set generated by $w$, that is the set $\{x\in\mathbb{I}: x=0.wv\ \text{for some}\ v\in\{0,1\}^\infty\}$. Notice that $[\epsilon]=[0,1]$.  We indicate the standard Lebesgue measure of a measurable set $A\subseteq \mathbb{I}$ by $m(A)$.

\begin{defn} Herein by \emph{simple substitution} we mean a map $\sigma:\lbrace0,1\rbrace \rightarrow \lbrace 0,1\rbrace^*$. For every integer $k\ge 2$, by \emph{$k$-block substitution} we mean a map $\sigma:\lbrace0,1\rbrace^k \rightarrow \lbrace 0,1\rbrace^*$.
\end{defn}
\noindent We only consider the case in which there is exactly one block mapped to empty word:
\begin{defn}
For every integer $k\ge 2$, by  \emph{erasing $k$-block substitution} we mean  a $k$-block substitution $\sigma $ such that there exists a unique $ w\in \lbrace0,1\rbrace^k : \sigma(w)=\epsilon$. 
\end{defn}

\noindent We will adopt the following conventions. The integer $k>2$ will always be the length of the blocks transformed by the erasing substitution $\sigma$. By $w_{(1)},w_{(2)},\dots,w_{(2^k)}$ we will indicate the lexicographic enumeration of $\{0,1\}^k$. For every $i\in\{1,\dots,2^k\}$ we set: 
\begin{equation} \label{sigma}
\sigma \left( w_{(i)}\right)=v_{(i)},\quad v_{(j)}=\epsilon,\quad v_{(i)} \in\{0,1\}^+ \text{ for } i\neq j 
\end{equation} 
We will always indicate by $w_{\epsilon}$ the word $w_j\in\{0,1\}^k$ such that $\sigma(w_j)=\epsilon$.

\begin{defn} \label{o.c}
We say that $\sigma$ verifies the \emph{optimality condition} (in short OC) if, for every $w\in \{0,1\}^\omega$, there exists a sequence of integers $(h_i)_{i\in\mathbb{N}}\in \{1,2,\dots,2^k\}^\mathbb{N}$ such that 
\begin{equation}
w=\prod_{i=1}^\infty v_{(h_i)}
\label{OC1}
\end{equation}
\end{defn}
\begin{rem}
Notice that OC implies also that, for every \emph{finite} binary word $w$, there is a suitable concatenation of nonempty words $v=v_{(h_1)}v_{(h_2)}\dots v_{(h_n)}$ such that $w$ is a prefix of $v$. 
\label{optimfinite}
\end{rem}
\noindent We will assume always $\sigma(\epsilon)=\epsilon$. The $k$-block substitution $\sigma$ induces then naturally a map over $\mathcal{W}_{k}=\left(\cup_{n\in\mathbb{N}_0}\{0,1\}^{nk}\right)\cup\{0,1\}^\omega$ if we set, for every $w\in \mathcal{W}_k$, $$\sigma(w)=\prod_{i=0}^{\frac{|w|}{k}-1} \sigma(w_{ki+1}w_{ki+2}\dots w_{k(i+1)}),$$

\noindent where the concatenation index has to be intended to be up to $\infty$ if $w\in\{0,1\}^\omega$. We can exploit this to define a function $f_\sigma:\mathbb{I}\rightarrow\mathbb{I}$, which is determined by the symbolic action of $\sigma$ on the binary expansion of real numbers in the unit interval. Assuming $w_\epsilon\ne 1^k$, we set:
\begin{equation}\label{fsigma__}
f_\sigma(x)=
\begin{cases}
\ \sum _{h=1}^{|\sigma(\widetilde{x})|} \frac{\left(\sigma(\widetilde{x})\right)_h}{2^h}=0.\sigma(\widetilde{x})\quad&\text{if}\ x\in(0,1]\ \text{and}\  \widetilde{x} \ne w^\infty_{\epsilon}\\
\\
\ 0&\text{if}\ \widetilde{x}=w^\infty_{\epsilon}\ \text{or $x=0$}
\end{cases}
\end{equation}
\noindent In the following, we will have to pay attention to two technicalities: the ambiguity in the binary representation of the elements of $\mathcal{Q}_2$ and the fact that the erasing character of $\sigma$ implies that the $\sigma$-image of some infinite words can be finite. Defining the set $$\mathcal{E}^{\sigma}=\{ x\in  \mathbb{I}\cap \mathbb{Q}:x=0.vw^\infty_{\epsilon}\ \text{for some}\ v\in\{0,1\}^{nk},\ n\in\mathbb{N}_0\},$$ 
it follows immediately that $|\sigma(\widetilde{x})|<\infty\implies x\in \mathcal{E}^{\sigma}$, and the converse implication also holds unless $w_\epsilon=0^k$.

\begin{rem}
If $w_\epsilon=1^k$, a complementary map $f_\sigma'$ can be introduced, replacing $\widetilde{x}$ by the unique infinite binary representation of $x\in\mathbb{I}$ not ending in $1^\infty$. In this case, the property of $f_\sigma'$ are analogous to that of $f_\sigma$ in the particular case $w_\epsilon=0^k$ (with the difference that the properties of left/right limits on $\mathcal{Q}_2$ for $f_\sigma$ correspond to properties of right/left limits for $f_\sigma'$). For this reason we will assume $w_\epsilon\ne 1^k$ throughout.
\end{rem}

\noindent A particular subset of $k$-block substitutions are those which can be rewritten using $k$ (generally distinct) simple substitutions, acting on the elements of a word $u$ according to the congruence class modulo $k$ of their indexes. More precisely, we introduce the following
\begin{defn}\label{defialter}
We say that the $k$-block substitution $\sigma$ is \textit{alternating} if there exist $k$ simple substitutions $\sigma_1,\sigma_2,\dots \sigma_k$ such that the map $\sigma_e$ defined below is an extension of $\sigma$ on $\{0,1\}^\infty$: 
\begin{equation}
\sigma_e(u)=\prod_{j=0}^{n-1} \left(\prod_{i=1}^k \sigma_i(u_{i+jk})\right)\prod_{i=1}^m \sigma_i(u_{i+nk})
\label{alternating}
\end{equation}
where $|u|=nk+m$ ($n,m\in\mathbb{N}_0$, $m<k$) and the first (last) product has to be taken as empty if $n=0$ ($m=0$). 
\end{defn}
\noindent Therefore, if $\sigma$ is alternating, we can use Eq.\eqref{alternating} to extend the definition of $\sigma$ to a morphism over all $\{0,1\}^{\infty}$ setting $\sigma\equiv \sigma_e$. We will exploit this fact in the following so that, when dealing with an alternating substitution $\sigma$, we will write simply $\sigma(u)$ (defined through the right hand side of Eq.\eqref{alternating}) for words of any length. Alternating substitution rules are quite well investigated and have also been generalized (see for instance \cite{culik1992iterative}). We will see that assuming $\sigma$ alternating makes significantly less complicated to deduce combinatorially its dynamical properties. However, for the sake of generality, we will not assume always this property. 
\begin{rem}
When we do not assume $\sigma$ alternating and $v\in\{0,1\}^{nk+m}$ \emph{(}$0<m<k$\emph{)}, we may write simply $\sigma(v)$ for the word $\sigma(v_1\dots v_{nk})$, i.e. we implicitly drop the last digits of $v$ so as to truncate it to a word in $\{0,1\}^{nk}$. Of course this crude sort of ``extension'' of $\sigma$ to $\{0,1\}^*$ is not a morphic map, and this convention will be used simply to lighten the notation in cases in which we are not interested in what happens after $\sigma(v_1\dots,v_{nk})$.
\label{convention}
\end{rem}

\noindent The next step consists in defining some properties of $\sigma$, strengthening it simply being erasing, which entail additional dynamical properties for the real map $f_\sigma$ associated to it. To do so we introduce some tools needed to settle words of any given length.

\noindent Let $w \in \lbrace0,1\rbrace^+$ be such that $n k< \left| w \right|<(n+1) k$ for some $n\in\mathbb{N}_0$. A word $u \in \lbrace0,1\rbrace^*$ such that $u=wv$ and $\left| u \right|=(n+1) k$ is said a \emph{k-rounding} of $w$; in this case $v$ is said a \emph{k-extension} of $w$. If $\left|w\right|=n k$ ($n\in \mathbb{N}$) we assume that $\epsilon$ and $w$ are, respectively, the unique $k$-extension and the unique $k$-rounding of $w$. We set $$e(w)= \lbrace v: v \mbox{ is a $k$-extension of }w\rbrace\ ,\ r(w)=\lbrace v: v \mbox{ is a $k$-rounding of }w\rbrace.$$

\noindent For $W\subseteq \lbrace 0,1 \rbrace^* $, we set $$e(W)=\bigcup_{w\in W} e(w)\quad, \quad r(W)=\bigcup_{w\in W} r(w).$$
Let $w\in \lbrace 0,1 \rbrace^+$. We set $w^{[0]}=\lbrace w\rbrace$ and, for any positive integer $n$:  

$$w^{[n]}=\lbrace \sigma(v): v\in r(w^{[n-1]})\rbrace .$$

\vspace{0.2cm}

\noindent Aimed at studying the dynamical properties of $f_\sigma$, we want to classify block substitutions based on how ``markedly erasing'' they are. More precisely, let us give the following  

\begin{defn} A $k$-block substitution $\sigma$ is said to be:
\begin{enumerate}
\item  \emph{Strongly erasing} if for every $w \in \lbrace0,1\rbrace^*$ there is $n \in \mathbb{N}$ such that:
\begin{equation}
 \epsilon\in w^{[n]} 
\label{str_eras}
\end{equation} 

\noindent So, when $\sigma$ is strongly erasing, for every $w\in\{0,1\}^*$ there exist a positive integer $n$ and finite words $r_{(0)},\ldots, r_{(n-1)}$, which we call \textit{erasing $k$-roundings of $w$}, such that: i) $r_{(0)}$ is a $k$-rounding of $w$;
ii) for any integer $j$ such that $1 \le j\le n-1$,  $r_{(j)}$ is a $k$-rounding of $\sigma(r_{(j-1)})$; iii) $\sigma(r_{(n-1)})=\epsilon $.
The $k$-extensions corresponding to the $n$ $k$-roundings just defined, i.e. the $n$ words $\{e_{(i)}\}_{i=0,\ldots, n-1}$, each belonging to $\{0,1\}^{< k}$, such that $we_{(0)}=r_{(0)}$ and, for $1\le i\le n-1$, $\sigma(r_{(i-1)})e_{(i)}=r_{(i)}$, will be called \textit{erasing $k$-extensions} for $w$. Notice that there can be more than one set of erasing $k$-roundings and of erasing $k$-extensions for a given word.

\item \emph{Completely erasing} if it is alternating (in the sense of Definition \ref{defialter}) and, for every $w\in\{0,1\}^*$, there is $n \in \mathbb{N}$ such that:
\begin{equation}
\sigma^n(w)=\epsilon
\label{com_eras}
\end{equation}
\noindent The smallest positive integer $n$ verifying Eq.\eqref{com_eras}, indicated by $\epsilon(w)$, is called \emph{vanishing order} of $w$.

\item \emph{Boundedly erasing} if $\sigma$ is completely erasing and $\epsilon(\cdot)$ is bounded over $\{0,1\}^*$.
\end{enumerate}
\end{defn}

\begin{lem}
For every $k$-block substitution $\sigma$, we have:
\begin{equation*}
\text{boundedly erasing}\implies \text{completely erasing} \implies \text{strongly erasing} \implies \text{erasing}
\end{equation*}
\label{classes}
\end{lem}
\begin{proof}
\noindent The first and last implications are trivial, so we just have to prove the middle one. Assume then that $\sigma$ is completely erasing. Since $\sigma$ is alternating, we can rewrite it by means of the simple substitutions $\sigma_1,\dots \sigma_k$ as in Eq.\eqref{alternating}. Applying $\sigma$ to $w_\epsilon$ we readily deduce that 
\begin{equation}
\sigma_i((w_\epsilon)_i)=\epsilon\ \text{ for every}\ i\in\{1,\dots,k\}.
\label{alterneras}
\end{equation}
\noindent Take then any $w\in \{0,1\}^*$. Set $w^{(0)}=w$  and write $|w|$ as $n_0k+m_0$, with $n_0\in\mathbb{N}_0$ and $m_0\in\{0,\dots, k-1\}$. Set $w^{(1)}=\sigma(w^{(0)}e^{(0)})$, where $e^{(0)}$ is the unique $k$-extension of $w^{(0)}$ such that $(e^{(0)})_i={(w_\epsilon)}_{m_0+i}$ for every $i	\in\{1,\dots,k-m_0\}$. Because of Eq.\eqref{alterneras}, we have that $\sigma(w^{(0)}e^{(0)})=\sigma(w^{(0)})$.
We can proceed in this way for further $\epsilon(w^{(0)})-1$ steps, setting, for every $j\in\{1,\dots,\epsilon(w)-1\}$: 
$$|w^{(j)}|=n_jk+m_j\ (n_j\in\mathbb{N}_0,\  m_j\in\{0,\dots, k-1\})\ , \  w^{(j+1)}=\sigma(w^{(j)}e^{(j)})$$
where  $|e^{(j)}|=k-m_j\ \text{and}\ \left(e^{(j)}\right)_i={(w_\epsilon)}_{m_j+i}\ \text{for every}\ i	\in\{1,\dots,k-m_j\}$. 

\noindent Recalling Eq.\eqref{alterneras}, at each step we have $\sigma(w^{(j)}e^{(j)})=\sigma(w^{(j)})$, so that
$$\sigma(w^{(\epsilon(w)-1)}e^{(\epsilon(w)-1)})=\epsilon,$$ which implies that $\epsilon\in w^{[\epsilon(w)]}$. 
\end{proof}

\noindent We now provide some examples of substitutions falling in each of the erasing classes introduced in Definition \ref{classes}. The substitution $\sigma_{1}$ defined below is erasing, but not strongly erasing.
Indeed, 110 produces a cycle of order 2 of words having length exactly $k$ (so that we do not have to choose any $k$-extension), that never reaches the empty word. Notice that $\sigma_1$ verifies the optimality condition \eqref{o.c}.

\noindent The substitution $\sigma_{2}$ defined below is strongly erasing, but not completely erasing.  Indeed, $\sigma_2$ is not alternating as, if so, odd-indexed 0s and even-indexed 1s should be mapped to $\epsilon$, which would imply, since $\sigma_2(00)=0$, that even-indexed 0s are mapped to 0, but this is incompatible with $\sigma_2(10)=1$. Notice then that $|\sigma_2(w)|\le (|w|)/2$ for $w$ having even length, while $|\sigma_2(w1)|\le (|w|+1)/2$ for $w$ having odd length, which implies that, $2$-extending when needed with 1, we arrive in a finite number of iterates at length 1. At that point, $2$-extending always with the digit 1 we arrive at $\epsilon$ in at most 2 steps. 
Notice that $\sigma_2$ verifies the optimality condition \eqref{o.c}.

\noindent The substitution $\sigma_{3}$ defined below is completely erasing, but not boundedly erasing.
Indeed $\sigma_3$ is alternating, with all 0s going to $\epsilon$, odd-indexed 1s going to $0$ and even-indexed 1s going to 1. Moreover $|(\sigma_3)^2(w)|$ is strictly less than $|w|$ for every finite word $w$, because all odd-indexed elements of $w$ go to $\epsilon$ in at most two iterations, which implies that $\sigma_3$ is completely erasing. On the other hand, $\sigma_3$ is not boundedly erasing as, taking for instance $w=1^{2m}$ ($m\in\mathbb{N}$), we have $|(\sigma_3)^2(w)|=1^m$, so that $\epsilon(w)$ will diverge with $|w|$. Notice that $\sigma_3$ verifies the optimality condition $\eqref{o.c}$.

\noindent Finally, the substitution $\sigma_{4}$ defined below is boundedly erasing. Indeed, it is alternating, as it can be rewritten as in Eq.\eqref{alternating} using the three simple substitutions $\sigma_4^{i}$ ($i=1,2,3$) defined by: $\sigma_4^{1,2,3}(0)=\epsilon$, $\sigma_4^1(1)=010$, $\sigma_4^2(1)=001$ and $\sigma_4^3(1)=000$. Observe that, for every word $w$ of length 3, $\sigma(w)$ is made of an integer number of words of length 3 which are all strictly smaller, in lexicographic order, than $w$. Since the smallest of these words, 000, is mapped to $\epsilon$, this easily implies that $\epsilon(w)\le 8$ for every $w\in\{0,1\}^*$.

\begin{tabular}{p{0.20 \textwidth} p{0.15\textwidth} p{0.17\textwidth}p{0.28\textwidth}}
\vspace{2mm}\\
$\sigma_1(000)=\epsilon$ &  $\sigma_2(00)=0$ & $\sigma_3(00)=\epsilon$ & $\sigma_4(000)=\epsilon$\\
$\sigma_1(001)=00$ &$\sigma_2(01)=\epsilon$ & $\sigma_3(01)=1$ & $\sigma_4(001)=000$ \\
$\sigma_1(010)=011$ & $\sigma_2(10)=1$ & $\sigma_3(10)=0$ & $\sigma_4(010)=001$\\
$\sigma_1(011)=010$ & $\sigma_2(11)=0$ & $\sigma_3(11)=01$ & $\sigma_4(011)=001000$\\
$\sigma_1(100)=10$ & & & $\sigma_4(100)=010$\\ 
$\sigma_1(101)=110$ & & & $\sigma_4(101)=010000$\\
$\sigma_1(110)=111$ & & & $\sigma_4(110)=010001$\\
$\sigma_1(111)=110$ & & & $\sigma_4(111)=010001000$\\
\\
\footnotesize{Erasing} & \footnotesize{Strongly} & \footnotesize{Completely} & \footnotesize{Boundedly}\\
& \footnotesize{$\, \,$erasing} & \footnotesize{$\ \ \,$erasing} & \footnotesize{$\ \ $erasing}
\vspace{2mm}
\end{tabular}
\vspace{0.4cm}
\begin{rem} The substitution $\sigma_3$ can be considered the simplest case of completely erasing substitution. The related map $f_{\sigma_3}$ is thoroughly studied in \cite{dellacorte2021simplest}.
\end{rem}

\noindent We will see that, moving towards completely erasing substitutions, we get increasingly complex dynamics for the corresponding interval map, while boundedly erasing substitutions have to be considered an extreme case, as they are so efficient in chopping off digits that the dynamical behavior becomes almost trivial.
\vspace{0.2cm}

\noindent In general an erasing block substitution $\sigma$ is a non-morphic map, in the following sense: if $w\in\{0,1\}^\omega$ and $u$ is a finite binary word, in general $\sigma(uw)\ne\sigma(u)\sigma(w)$. To address this point, let us introduce a countable family of maps induced by an erasing $k$-block substitution $\sigma$ and by the choice of a finite word $u$.
 
\noindent Take $u\in\{0,1\}^*$, $v\in\{0,1\}^\infty$.  Let us assume that $\sigma^0_u(v)=\sigma^0(v)=v$. Let us define, then,  the map $\sigma_u:\{0,1\}^\infty\to \{0,1\}^\infty$ by means of the equality 
\begin{equation}
\sigma(uv)=\sigma(uu^{(1)})\sigma_u(v),
\label{sigmau}
\end{equation}
where $u^{(1)}$ is the unique prefix of the word $v$ which is a $k$-extension of $u$. Set $\sigma_u^1(w)=\sigma_u(w)$. Notice that here, at the left hand side of Eq.\eqref{sigmau}, we apply the convention described in Remark \ref{convention}, i.e. we ``drop'' the digits of $v$ that exceed the longest of the prefixes of $uv$ having length multimple of $k$. 

\noindent Iterating the procedure, for every positive integer $n$, we wet $p^{(0)}=u$ and define the words $p^{(n)}$ and the maps $\sigma_u^n:\{0,1\}^\infty\to\{0,1\}^\infty$ inductively by means of the following equalities:
\begin{equation}
\begin{cases}
p^{(1)}=\sigma \left( uu^{(1)}\right)\\
\sigma^2(uv)=\sigma \left( p^{(1)}u^{(2)}\right)\sigma_u^2(v)\\
\dots\\ 
\dots\\
p^{(n-1)}=\sigma \left( p^{(n-2)}u^{(n-1)}\right)\\
\sigma^n(uv)=\sigma \left( p^{(n-1)}u^{(n)}\right)\sigma_u^n(v)
\end{cases}
\label{sigma_u^n}
\end{equation}
where, at each stage, $u^{(i)}$ is the unique prefix of $\sigma_u^{i-1}(v)$ which is a $k$-extension of $p^{(i-1)}$. Notice that $\sigma_{\epsilon}$ coincides with $\sigma$.
\begin{rem}
Here we apply the convention described in Remark \ref{convention}, i.e. we ``drop'' the digits of $uv$ whenever this is necessary to get a word having length multiple of $k$ ensuring that $\sigma^i(uv)$ makes sense for $1\le i\le n$. As observed before, this convention is used simply to lighten the notation in cases in which we are not interested in what happens after $\sigma_u^i(v')$. 
\label{lighten}
\end{rem}

\noindent In case $\sigma$ is alternating there is no need of a convention of the type described in Remarks \ref{convention} and \ref{lighten}, because in that case, thanks to Eq.\eqref{alternating}, $\sigma(w)$ already makes sense for words of any length, and therefore a simpler definition can be given for $\sigma_u$. Indeed, for every positive integer $n$ and every $v\in\{0,1\}^\infty$, we define $\sigma^n_u(v)$ as the word verifying
\begin{equation}
\sigma^n(uv)=\sigma^n(u)\sigma^n_u(v)
\label{sigma_u^n_compl}
\end{equation}
\noindent Notice finally that, both in Eqs.\eqref{sigma_u^n} and in Eq.\eqref{sigma_u^n_compl}, the map $\sigma^n_u$ does not coincide with $(\sigma_u)^n$, that is the $n$-th iterate of the map $\sigma_u$. Instead, $\{\sigma^n_u\}_{n\in\mathbb{N}}$ is a countable set of generally distinct maps indexed by the positive integers. 

\section{General properties of the map $f_\sigma$}
\label{generall}
\noindent In this Section we want to address some analytical properties verified by $f_\sigma$ when $\sigma$ is an erasing $k$-block substitution. \textbf{Throughout the Section, $\sigma$ is a $k$-block erasing substitution defined by Eq.\eqref{sigma}.}
\begin{lem}
If $\sigma$ verifies the optimality condition \eqref{o.c}, then $f_\sigma:\mathbb{I}\to\mathbb{I}$ is onto.
\label{surj}
\end{lem}
\begin{proof}
Take $y\in \mathbb{I}$. By the optimality condition, there exists a sequence $h=h_1,h_2,\ldots \in \lbrace1,2,\ldots ,2^k \rbrace^\mathbb{N}$ such that
$u=\prod_{i=1}^\infty v_{(h_i)}$, $y=0.u$ and $v_{(h_i)}\ne\epsilon$ for every $i\in\mathbb{N}$.

\noindent Suppose that $w=\prod_{i=1}^\infty w_{(h_i)}$ does not end in $0^\infty$.
Then we can take $x\in \mathbb{I}$ such that
\begin{equation}\widetilde{x}=\prod_{i=1}^\infty w_{(h_i)}
\label{xtilda}
\end{equation}
so that
$\sigma(\widetilde{x})=\prod_{i=1}^\infty \sigma \left(w_{(h_i)}\right)=\prod_{i=1}^\infty v_{(h_i)}=u$.
Since $\sigma(\widetilde{x})=u\in \lbrace0,1 \rbrace^\omega$, we have $x\notin \mathcal{E}^\sigma$. It follows that
$f_\sigma(x)=0.\sigma(\widetilde{x})=0.u=y.$

\noindent Suppose instead that $w$ ends in $0^\infty$ and let $a\in\{0,1\}^*$ be the shortest word such that $w=a0^\infty$ (notice that this implies that $w_{\epsilon}\ne 0^k$, as $w_{\epsilon}$ does not appear in the right hand side of \eqref{xtilda}). Then we cannot find any $x\in\mathbb{I}$ such that $\widetilde{x}=w$. However, we can insert infinitely many times $w_\epsilon$ in $w$ as a $k$-factor, obtaining a point whose image is $y$. More precisely, set: $S_y=\{b\in\{0,1\}^\omega: b=a\prod_{i=1}^\infty 0^{n_i}w_\epsilon^{m_i}\}$, 
where $(m_i)_{i\in\mathbb{N}}$ and $(n_i)_{i\in\mathbb{N}}$  are sequences of non-negative integers both not ultimately 0, and $|a|+\sum_{j=1}^{i}n_j$ is a multiple of $k$ for every $i\in\mathbb{N}$. Recalling that $\sigma(w_\epsilon)=\epsilon$, it follows that $\emptyset\ne\{x\in\mathbb{I}:\widetilde{x}\in S_y\}\subseteq f_\sigma^{-1}(y)$.
\end{proof}
\noindent In the following it will be useful to have a symbol for the points considered in the second part of the previous proof, so we set 
$$\mathcal{F}^\sigma=\{y\in\mathbb{Q}\cap\mathbb{I} : y=0.wv_{(0)}^\infty,\  w=\prod_{j=1}^n v_{(i_j)}\ \text{for some $n\in\mathbb{N}_0$, and}\ v_{(0)}=\sigma(0^k)\}.$$ 

\begin{lem}
The map $f_\sigma$ is continuous on the set $\mathcal{C}^\sigma=\mathbb{I}\setminus (\mathcal{Q}_2^0\cup \mathcal{E}^\sigma)$,  and therefore Baire-1. Moreover,  $f_\sigma$ is left-continuous on $\mathbb{Q}_2\setminus\{0\}$.
\label{conti}
\end{lem}
\begin{proof}
Take $x\in \mathcal{C}^\sigma$ and fix $n\in \mathbb{N}$ . Then $\widetilde{x}=\prod_{h=1}^\infty w_{(i_h)}$ ($i_h\in \lbrace1,\ldots2^k\rbrace$) and $\widetilde{x}$ does not have $w_\epsilon^\infty$ as a $k$-factor.  Since $|v_{(j)}|\ge 1$ if $w_{(j)}\ne w_\epsilon$, there is a prefix $p$ of $\widetilde{x}$ such that $\left| p \right|=m k$ ($m\in \mathbb{N}$) and $\left| \sigma(p)\right| \ge n$. Since $x$ has a unique binary expansion, for every $y\in\mathbb{I}$ such that $\left|x-y \right|<2^{-mk}$, we have that $\widetilde{x},\widetilde{y}$ coincide at least up to the first $mk$ digits. Then
$\left| f_\sigma(x) - f_\sigma(y)\right|\le 2^{-(\left|\sigma(p)\right|)}\le 2^{-n}$. Since we assume $w_{\epsilon} \neq 1^k$,  recalling that $\sigma$ acts on the 1-periodic expansion of dyadic rationals, the argument above shows that $f_\sigma$ is left continuous on $\mathbb{Q}_2\setminus\{0\}$ (notice that, assuming instead $w_{\epsilon} \ne 0^k$, we would get right continuity on $\mathbb{Q}_2\setminus \{1\}$). 

\noindent Finally, interval maps with countably many discontinuities are pointwise limit of continuous functions, that is Baire-1 (see for instance \cite{van1982second}).
\end{proof}
\noindent Notice that, in the particular case $w_\epsilon=0^k$, the sets $\mathcal{E}^\sigma$, $\mathcal{F}^\sigma$, $\mathcal{Q}_2^0$ and $\mathbb{I}\setminus \mathcal{C}^\sigma$ coincide.
\begin{rem}
For every pair of integers $m,n>0$ there are irrational points $x$ such that $\widetilde{x}$ has a prefix of type $pw_\epsilon^m$ \emph{(}$p\in\{0,1\}^{nk}$\emph{)}. Two points sharing such a prefix are arbitrarily close if $m\to\infty$, but their $f_\sigma$-images are not necessarily arbitrarily close, because $|\sigma(pw_\epsilon^m) |$ is independent of $m$. Therefore $f_\sigma\vert_\mathcal{C^\sigma}$ is in general not uniformly continuous, so that it is impossible to extend it to a continuous function on $\mathbb{I}$.
\end{rem}

\noindent We recall that an interval map $g:\mathbb{I}\to \mathbb{I}$ is called \textit{Darboux} if the image of every interval is an interval. From the point of view of combinatorial dynamics, Darboux functions have some of the properties of continuous functions, in particular for what concerns the existence of periodic points (see e.g. \cite{korczak2015topological}, where Darboux plus Baire-1 is shown to imply the Sharkowsky ordering of periodic points).

\noindent However, the general interval map $f_\sigma$ is not Darboux unless $\sigma$ is constructed quite \textit{ad hoc}. Take indeed $y\in\mathcal{Q}_2$ such that  $\widetilde{y}=\left(\prod_{n=1}^N w_{(i_n)}\right) 1^\infty$ for some  $N\in\mathbb{N}$ and $w_{i_N}$ ends with 0. 
Let $p\in\{1,\dots,2^k\}$ be such that $w_{(p)}\in\{0,1\}^k$ is the unique word of length $k$ coinciding up to the digit $k-1$ with $w_{(i_N)}$, but ending with 1. Supposing that $w_\epsilon\ne 0^k$, the right and left limit of $f_\sigma$  at $y$ are written as: 
\begin{equation}
\lim_{x\to y^-} f_\sigma(x)=f_\sigma(y)=0.\left(\prod_{n=1}^N v_{(i_n)}\right)(\sigma(1^k))^\infty
\label{lim1}
\end{equation}
 \begin{equation}
 \lim_{x\to y^+} f_\sigma(x)=0.\left(\prod_{i=1}^{N-1} v_{(i_n)}\right)v_{(p)} (\sigma(0^k))^\infty
\label{lim2}
\end{equation} 
Clearly $f_\sigma$ can be Darboux only if these two real numbers coincide for every $y\in\mathcal{Q}_2$, which will be not true in general. The map $f_\sigma$ is in general not Darboux also when $w_\epsilon=0^k$. The simplest example to see that is given by the substitution $\sigma_3$ defined in Section \ref{prelim}, for which, when $y\in\mathcal{Q}_2$, the left limit $\lim_{x\to y^-}f_{\sigma_3}(x)=f_{\sigma_3}(y)$ does not belong to $$[\liminf_{x\to y^+}f_{\sigma_3}(x), \limsup_{x\to y^+}f_{\sigma_3}(x)]$$ (see \cite{dellacorte2021simplest} for further details).
Therefore, for our general object, we cannot expect a priori any relation between periodic points of different order, so that we will have to provide an explicit proofs of basic combinatorial dynamical properties. To proceed further, let us introduce some technical tools. 

\begin{defn} 
Let $\Xi\subset \mathbb{N}_0^{\mathbb{N}}$ be an uncountable set of sequences such that
\begin{equation}
a,b\in\Xi, a\ne b\implies \forall M>0,\ \exists n\in\mathbb{N}_0: n>M\ \text{and}\ a_n\ne b_n
\label{defiXi}
\end{equation}
Set $\Omega=\Xi\cap \{0,1\}^{\mathbb{N}}$ and call  $\Xi^+$ $(\Omega^+)$ the uncountable subset of $\Xi$ ($\Omega$) such that no one of its elements is ultimately 0. 
\label{defiomegaxi}
\end{defn}

\begin{defn}
Take $y\in\mathbb{I}\setminus \{0\}$ and write $\widetilde{y}=\prod_{i=1}^\infty w_{(h_i)}$. Let $(h_{i_j})_{j\in \mathbb{N}}$ be the subsequence of all $h_i$ such that $w_{(h_{i_j})}\ne w_\epsilon$. Then we define:
\begin{equation}
y^\circ= \xi\in\mathbb{I}\ \text{such that}\  \widetilde{\xi}=\prod_{i=1}^\infty u_{(j)},  \text{where } u_{(j)}=w_{(h_{i_j})}
\end{equation}
\label{circ}
\end{defn}
\noindent In other words, $y^\circ$ is the point whose binary expansion (ending in $1^\infty$, if there is a choice) coincides with that of $y$ after having erased all the (irrelevant, from the point of view of $\sigma$) occurrences of $w_{\epsilon}$ which are $k$-factors, so that $\widetilde{y^\circ}_{nk+1}\ldots\widetilde{y^\circ}_{(n+1)k} \neq w_{\epsilon}$ for every non-negative integer $n$. 

\begin{defn}
\noindent For $u\in\{0,1\}^\omega$ such that $u=\prod_{i=1}^\infty u_{(i)}$ (with $u_{(i)}\in\{0,1\}^k$), and for every $a\in\Xi$, we set:
\begin{equation}
\xi_{u}(a)=0.\prod_{i=1}^{\infty}w_\epsilon^{a_i}u_{(i)}
\label{mapxi}
\end{equation}
\label{xxi}
\end{defn}

\begin{lem}
For every $x\in \mathcal{S}^\sigma=f_\sigma(\mathbb{I})\setminus\mathcal{Q}_2^0$, the set $f_\sigma^{-1}(x)$ is uncountable.
\label{uncount} 
\end{lem}
\begin{proof}
\noindent Suppose first that $x\in \mathcal{S}^\sigma\setminus \mathcal{F}^\sigma$. Then there is $w\in\{0,1\}^\omega$ such that $w$ does not end in $0^\infty$ and $x=0.\sigma(w)$, so that there exists $y\in\mathbb{I}$ such that $w=\widetilde{y}$, and thus $y\in f_\sigma^{-1}(x)$. Moreover $\widetilde{y}$ does not have $w_\epsilon^\infty$ as a $k$-factor, because if so then $|\sigma(w)|$ would be finite so that $x\in\mathcal{Q}_2^0$ against the hypothesis.
Since every occurrence of $w_\epsilon$ as a $k$-factor is mapped by $\sigma$ to the empty word, we have $f_\sigma(y^\circ)=f_\sigma(y)$, so that $y^\circ\in f_\sigma^{-1}(x)$.

\noindent Let us write $y^\circ$ as $\prod_{i=1}^\infty w_{(h_i)}$ ($w_{(h_i)}\ne w_\epsilon$ for every $i$). By definition of the map $\xi$, we have that, for every $a\in \Xi$, $y_{(a)}:=\xi_{\widetilde{y^\circ}}(a)=0.\prod_{i=1}^{\infty}w_\epsilon^{a_i}w_{(h_i)}$, which implies that every point $y_{(a)}$ ($a\in\Xi$) is in $f_\sigma^{-1}(x)$. Recalling the definition of $\Xi$, it follows that the $\xi_{\widetilde{y^\circ}}$-image of every sequence $a\in\Xi$ is obtained by $\widetilde{y^\circ}$ inserting, infinitely many times, finite and nonempty concatenations of $w_\epsilon$ with itself, so that we are ensured that, for $u=\widetilde{y^\circ}$, the product in the right hand side of Eq.\eqref{mapxi} will never end with $0^\infty$ (both if $w_\epsilon=0^k$ or not). Since in the binary expansion of $y^\circ$ there is never $w_\epsilon$ as a $k$-factor, it follows that $\xi_{\widetilde{y^\circ}}$ is injective, which implies that $f_\sigma^{-1}(x)$ is uncountable.

\noindent Suppose now that $x\in \mathcal{S}^\sigma\cap \mathcal{F}^\sigma$. Since $x\in\mathcal{F}^\sigma$, there exists a word $w\in\{0,1\}^\omega$ ending in $0^\infty$ such that $0.\sigma(w)=x$. Notice that $x\in \mathcal{S}^\sigma$ implies that $|\sigma(w)|=\infty$, so that we can exclude $w_\epsilon=0^k$. Consider now the restriction $\xi^+_w$ to $\Xi^+$ of the map $\xi_w$ defined in Definition \ref{xxi}. Since every $a\in \Xi^+$ is not ultimately 0, we have that $$w_{(a)}:=\prod_{i=1}^{\infty}w_\epsilon^{a_i}w_{(h_i)}$$ will never end in $0^\infty$, so that $w_{(a)}=\widetilde{\xi_w(a)}$. Moreover, $\sigma(w_{(a)})=\sigma(w)$, so that $\xi_w(a)$ belongs to $f_\sigma^{-1}(x)$ for every $a\in\Xi^+$. The same argument as above shows that $\xi^+_w$ is injective, which concludes the proof. 
\end{proof}

\begin{corol}
If $\sigma$ verifies the optimality condition \ref{o.c}, then for every $x\in\mathbb{I}$ the set $f_\sigma^{-1}(x)$ is uncountable.
\label{corol_uncount}
\end{corol}
\begin{proof}
Under the assumption, for every $x\in\mathbb{I}$ there are an infinite binary word $u$ and $y\in\mathbb{I}$ such that $0.\sigma(\widetilde{y})=0.u=x$. Then the thesis easily follows from the argument used in the proof of the previous Lemma.
\end{proof}

\noindent A relatively simple characterization of the topological structure of the $f_\sigma$-preimages of singletons can be achieved strengthening the optimality condition. Even so, there are a few technical issues, the main being that the points belonging to $\mathcal{Q}_2$ can be ``reached'' in general in infinitely many ways. Indeed, each of them has two infinite binary expansions plus a countable family of finite ones (all ending in a finite string of 0s except at most one) which, from a symbolic point of view, constitute distinct representations. The different cases are covered in the following Lemma, which is the main technical tool needed to prove that, outside $\mathcal{Q}_2$, the closures of the fibers of points verifying the strengthened version of the optimality condition are Cantor sets. 

\begin{lem}
Suppose that $\sigma$ verifies the optimality condition \ref{o.c} and that, for the point $y\in\mathbb{I}$, the words $w\in\{0,1\}^\omega$ verifying $y=0.w$ are such that the sequence $h_1,h_2\dots$ satisfying Eq.\eqref{OC1}, and  such that $v_{(h_i)}\ne\epsilon$ for every $i\in\mathbb{N}$, is unique.
If $x$ is a limit point of $f_\sigma^{-1}(y)$, then one of the following cases occurs: 
\begin{enumerate}
\item $x\in f_\sigma^{-1}(y)$
\item $x=0.vw_{\epsilon}^\infty$, with $|v|=nk$ for some non-negative integer $n$.
\item $y\in\mathcal{F}^\sigma$ and $x=0.a\prod_{i=1}^\infty 0^{n_i}w_\epsilon^{m_i}$, 
where $a\in\{0,1\}^*$, $\{m_i\}$ is a sequence of non-negative integers which is ultimately zero and $n_i$ are positive integers such that $|a|+\sum_{j=1}^{i}n_j$ is a multiple of $k$ for every $i\in\mathbb{N}$.
\end{enumerate} 
\label{limit_}
\end{lem}
\begin{proof}
\noindent We divide the proof in four cases.
\begin{enumerate}
\item $y\notin \mathcal{Q}_2\cup\mathcal{F}^\sigma$. 

\noindent Then $y$ has a unique binary expansion $u$ and, by hypothesis, there exists a unique sequence $h_1, h_2,\dots$ such that  $u=\prod_{i=1}^\infty v_{(h_i)}$. We can assume $v_{(h_i)} \ne\epsilon$ for every $i\in\mathbb{N}$. Moreover, since $y\notin \mathcal{F}^\sigma$, $w=\prod_{i=1}^\infty w_{(h_i)}$ does not end in $0^\infty$. Therefore, the point $x=0.w$ belongs to $f_\sigma^{-1}(y)$ and it coincides with $x^\circ$. By uniqueness of the sequence $h_1, h_2\dots$, it follows that $f_\sigma^{-1}(y)$ can be written as

\begin{equation}
\bigcup_{a\in {\mathbb{N}_0}^{\mathbb{N}}} \xi_{\widetilde{x^\circ}} (a)
\end{equation}
where $\xi_{\widetilde{x^\circ}}:{\mathbb{N}_0}^{\mathbb{N}}\to \mathbb{I}$ is the map defined in \ref{xxi}. 
If $z$ is a limit point of $f_\sigma^{-1}(y)$, there exists a sequence $\lbrace a^{(n)} \rbrace_{n\ge1}$  of elements of ${\mathbb{N}_0}^{\mathbb{N}}$ such that $\xi_{\widetilde{x^\circ}}(a^{(n)})\xrightarrow[n \to \infty]{} z$. Notice that, since $\widetilde{x^\circ}$ does not end with $0^\infty$, neither does $\widetilde{\xi_{\widetilde{x^\circ}}}(a)$ for every $a\in\mathbb{N}_0^\mathbb{N}$. 

\noindent We indicate by $ a_{i}^{(n)}$ the $i$-th element of the sequence $ a^{(n)}$.  Let us first prove that, for every $i\in\mathbb{N}$, there exists $a_i=\lim_{n\to\infty} a_i^{(n)}\le\infty$. Indeed, suppose that $m$ is the smallest positive integer such that $a_m^{(n)}$ admits two distinct sublimits, i.e. such that there exist two subsequences $n_h$ and $n_j$ such that $a_m^{(n_h)}\to\overline{a_m}$ and $a_m^{(n_j)}\to\overline{\overline{a_m}}$. Set
$q=w_\epsilon^{a_1}w_{(h_1)}\dots w_\epsilon^{a_{m-1}}w_{(h_{m-1})} w_\epsilon^{\min\{\overline{a_i},\overline{\overline{a_i}}\}}$ and $\delta=|\overline{a_i}-\overline{\overline{a_i}}|$. Then the points 
\begin{equation}
\{\xi_{\widetilde{x^\circ}}(a^{(n_h)})\}_{h\in\mathbb{N}},\ \{\xi_{\widetilde{x^\circ}}(a^{(n_j)})\}_{j\in\mathbb{N}}
\label{points_}
\end{equation}
ultimately belong respectively to the (closed) intervals $[qw_\epsilon^\delta w_{(h_m)}]$ and $[qw_{(h_m)}]$, which, recalling that $w_\epsilon\ne w_{h_m}$, have nonempty intersection only if either 1) $\delta=0$ or 2) $\delta\in\mathbb{N}$ and $w_{(h_m)}=p0$ and $w_\epsilon=p1$, or alternatively $w_{(h_m)}=p1$ and $w_\epsilon=p0$ for some $p\in\{0,1\}^{k-1}$. We can assume the former (in the other case the reasoning is completely analogous). A point $z$ in $[qw_\epsilon^\delta w_{(h_m)}]\cap [qw_{(h_m)}]$ must then admit the two binary forms $z=0.qp10^\infty$ and $z=0.qp01^\infty$. The first one can be only achieved if $\delta=1$, $w_{(h_m)}=0^k$ and $w_{(h_j)}=w_{(h_m)}$ for all $j>m$. This in turn implies that the second binary form, ending in $1^\infty$, can be only obtained if $w_\epsilon=1^k$.  But then $w_{(h_m)}$ and $w_\epsilon$ differ on every digit, which contradicts the existence of the word $p$ required before. It follows that it has to be $\delta=0$, which implies $\overline{a_i}=\overline{\overline{a_i}}$.

\noindent Assume now that it is nonempty the set $A$ of $i\in \mathbb{N}$ such that $a_i=\infty$, and set $z=0.u$. 
Then 
\begin{equation}
u=pw^\infty_{\epsilon},
\label{limitpoint}
\end{equation}
where $p\in \lbrace0,1\rbrace^*$ is the word $p=w_{\epsilon}^{a_1}w_{(h_1)}w_{\epsilon}^{a_2}w_{(h_2)}\ldots w_{\epsilon}^{a_{\bar{i}-1}}w_{(h_{\bar{i}-1})}$
with $\bar{i}=\min A$. Therefore, it falls under point 2. in the statement of the Lemma.

\noindent Suppose instead that $\lim_{n\to\infty} a^{(n)}_i\ne\infty$ for every integer positive $i$. Then $u=\prod_{i=1}^\infty w_{\epsilon}^{a_i}w_{(h_i)}$, which implies that $z\in f_\sigma^{-1}(y)$. Therefore, it falls under point 1. in the statement of the Lemma. Notice also that in this case, since $u$ is obtained from $\widetilde{x}$ by inserting (infinitely many times) finite concatenations of $w_\epsilon$ with itself in a word not having $w_\epsilon$ as a $k$-factor, $u$ cannot end in $0^\infty$ either if $w_\epsilon=0^k$ or if not, so that $u=\widetilde{z}$.

\item $y\in \mathcal{Q}_2\setminus\mathcal{F}^\sigma$.

\noindent By hypothesis $y$ has two binary expansions, $\widetilde{y}$ and $u\in\{0,1\}^\omega$ such that $u=v0^\infty$.
Then: i) by assumption there exist two words $w^{(1)},w^{(2)}\in\{0,1\}^\omega$ not having $w_\epsilon$ as a $k$-factor, such that $\sigma(w^{(1)})=\widetilde{y}$ and $\sigma(w^{(2)})=u$; ii) there is a countable family of finite integer sequences $$J=\cup_{j\in\mathbb{N}_0} \{j_i\}_{i=1\dots m_j}$$ such that there are words $w^{(3)}_{(j)}=\prod_{i=1}^{m_j} w_{(j_i)}$ ($m_j\in\mathbb{N}$) having the property that $\sigma\left(w^{(3)}_{(j)}\right)=v0^{g_j}$ for suitable non-negative integers $g_0,g_1,\dots\in \mathbb{N}_0$. Notice that the family $J$ is infinite, because the optimality condition ensures that there are arbitrarily large integers $h$ such that we can write $v0^h$ as a concatenation of the words $v_{(i)}=\sigma(w_{(i)})$. Recalling the definition of $f_\sigma$ (Eq. \eqref{fsigma__}), it follows that any point of type $0.w^{(3)}_{(j)}$ belongs to $f_\sigma^{-1}(y)$.

Then the set $f_\sigma^{-1}(y)$ decomposes as $S_1\cup S_2 \cup S_3$, where 
\begin{eqnarray*}\label{fibercase2}
&S_1=\bigcup_{a\in {\mathbb{N}_0}^{\mathbb{N}}} \xi_{w^{(1)}} (a)\ \ ,\ \ S_2=\bigcup_{a\in {\mathbb{N}_0}^{\mathbb{N}}} \xi_{w^{(2)}} (a)\\
&\\
&S_3= \bigcup_{a\in \mathbb{N}_0^{m}, j\in\mathbb{N}_0} \{x\in \mathbb{I} : x=0.\left(\prod_{i=1}^{m_j}w_\epsilon^{a_i}w_{(j_i)}\right)w_\epsilon^\infty\},
\end{eqnarray*}

where $\xi_{w^{(i)}}$ ($i=1,2$) is the map defined in \ref{xxi}. Since a limit point of a finite union of sets is limit point of (at least) one of the sets, we can repeat the argument used in Case 1. separately for $S_1$ and $S_2$. A similar argument shows that also the set $S_3$ has always limit points of the form \eqref{limitpoint}.

\item $y\in \mathcal{F}^\sigma\setminus \mathcal{Q}_2$.

\noindent Then there is no $x\in\mathbb{I}$ such that $f_\sigma(x)=y$ and $\widetilde{x}$ does not have $w_\epsilon$ as a $k$-factor. Indeed, by definition of $\mathcal{F}^\sigma$, there is $a\in\{0,1\}^*$ such that $0.\sigma(a0^\infty)=y$. Recalling that $f_\sigma$ is defined by the action of $\sigma$ on the 1-periodic expansion of dyadic rationals, it follows that 
\begin{equation}x\in f^{-1}_\sigma(y)\implies x=0.a\prod_{i=1}^\infty 0^{n_i}w_\epsilon^{m_i}
\label{0a}
\end{equation}
where $\{m_i\}$ is a sequence of non-negative integers which is not ultimately zero and $n_i$ are positive integers such that $|a|+\sum_{j=1}^{i}n_j$ is a multiple of $k$ for every $i\in\mathbb{N}$.
Let us define the set $\mathcal{B}$ as:
\begin{equation}
\mathcal{B}=\{b\in\{0,1\}^*: b=a\prod_{i=1}^N 0^{n_i}w_\epsilon^{m_i}\ \text{for some $N\in\mathbb{N}_0$}.\}
\label{b_}
\end{equation}  
Recalling Case 1., it follows that the set  
\begin{equation}
\mathcal{B}^\sigma_1=\{0.bw_\epsilon^{\infty} ,\  b\in\mathcal{B}\}
\end{equation}

belongs to $\overline{f_\sigma^{-1}(y)}$,  and its points fall under point 2. in the statement of the Lemma.  Moreover,  it follows that the set
\begin{equation}
\mathcal{B}^\sigma_2=\{0.b0^{\infty} ,\  b\in\mathcal{B}\}
\end{equation}
belongs to $\overline{f_\sigma^{-1}(y)}$,  and its points fall under point 3. in the statement of the Lemma. 

\item $y\in \mathcal{Q}_2\cap\mathcal{F}^\sigma$.

Suppose first that $w_\epsilon\ne 0^k$. By assumption $y$ has two infinite binary expansions, $\widetilde{y}$ and $u\in\{0,1\}^\omega$ such that $u=v0^\infty$ for some finite word $v$. 
Since $y\in \mathcal{F}^\sigma$ and is a dyadic rational, we have that $\sigma(0^k)=0^{N_1}$ or $\sigma(0^k)=1^{N_2}$ for some $N_1,N_2 \in \mathbb{N}$. In the following we consider the first case, a similar argument applies in the other one switching the roles of $\widetilde{y}$ and $u$.  

\noindent Then: i) there exists a word  $w^{(1)} \in \{0,1\}^\omega$ not having $w_\epsilon$ as a $k$-factor, such that $\sigma(w^{(1)}) = \widetilde{y}$; ii) moreover, recalling that $f_\sigma$ is defined by the action of $\sigma$ on the 1-periodic expansion of dyadic rationals and that $\sigma(0^k)$ is a string of 0s, there is no $x\in\mathbb{I}$ such that $\sigma(\widetilde{x})=u$ and $\widetilde{x}$ does not have $w_\epsilon$ as a $k$-factor, while, by definition of $\mathcal{F}^\sigma$, there is $c\in\{0,1\}^*$ such that $\sigma(c0^\infty)=u=v0^\infty$; iii) finally, there is a countable family of finite integer sequences $J=\cup_{j\in\mathbb{N}_0} \{j_i\}_{i=1\dots m_j}$ such that there are words $w^{(2)}_{(j)}=\prod_{i=1}^{m_j} w_{(j_i)}$ ($m_j\in\mathbb{N}$) having the property that $\sigma\left(w^{(2)}_{(j)}\right)=v0^{g_j}$ for suitable non-negative integers $g_j\in \mathbb{N}_0$.  It follows that $f_\sigma^{-1}(y)$ decomposes as $S_1 \cup S_2 \cup S_3$, where
\begin{eqnarray*}
&S_1=\bigcup_{a\in {\mathbb{N}_0}^{\mathbb{N}}} \xi_{w^{(1)}} (a)\ \ ,\ \ S_2=\{x\in \mathbb{I} : x=0.c\prod_{i=1}^\infty 0^{n_i}w_\epsilon^{m_i}\} \\
 &\\
  &S_3= \bigcup_{a\in \mathbb{N}_0^{m},j\in \mathbb{N}_0} \{x\in \mathbb{I} : x=0.\left(\prod_{i=1}^{m_j}w_\epsilon^{a_i}w_{(j_i)}\right)w_\epsilon^\infty\},
\label{fibercase4}
\end{eqnarray*}
where $\{m_i\}$ and $\{n_i\}$ have the same properties of Case 3. and $\xi_{w^{(1)}}$ is the map defined in \ref{xxi}.  Then we can repeat the argument of Case 2. for $S_1$ and $S_3$ and the argument of Case 3. for $S_2$. 

Finally, notice that, if $w_\epsilon=0^k$, the sets $S_2$ and $S_3$ do not belong to $f_\sigma^{-1}(y)$, since $\sigma$ acts on the 1-periodic expansion of dyadic rationals. 
\end{enumerate} \end{proof}
\begin{thm}
In the hypotheses of Lemma \ref{limit_}, the closure of $f_\sigma^{-1}(y)$ is a Cantor space, as a topological subspace of  $\mathbb{I}$, if $y\notin \mathcal{Q}_2$. 
\label{cantorr}
\end{thm}
\begin{proof}
Suppose $y\notin \mathcal{Q}_2\cup \mathcal{F}^\sigma$ and take the point $x^\circ\in f_\sigma^{-1}(y)$. 
By the previous Lemma, $f_\sigma^{-1}(y)=\bigcup_{ a\in \mathbb{N}_0^{\mathbb{N}}} \xi_{\widetilde{
x^\circ}}(a)$, where $\xi_{\widetilde{x^\circ}}:{\mathbb{N}_0}^{\mathbb{N}}\to \mathbb{I}$  is the map defined in \ref{xxi}, and
$Q=\overline{f_\sigma^{-1}(y)}\setminus f_\sigma^{-1}(y)$ 
is a countable set.  If $a,b\in \mathbb{N}_0^{\mathbb{N}}$ agree up to the $n-$th element, then $|\xi_{\widetilde{x^\circ}}(a)-\xi_{\widetilde{x^\circ}}(b)|\le 2^{-nk-\sum_{i=1}^n a_ik}$, which implies that $f_\sigma^{-1}(y)$ is dense in itself, and since the points in $Q$ are limit points of $f_\sigma^{-1}(y)$ by construction, so is $\overline{f_\sigma^{-1}(y)}$. Moreover if, in the binary expansion $w_{\epsilon}^{a_1}w_{(h_1)}w_{\epsilon}^{a_2}w_{(h_2)}\dots$ of $\xi_{\widetilde{x^\circ}}(a)$, we replace $w_{(h_m)}$ by a word $v\in\{0,1\}^k$ which is distinct from both $w_{\epsilon}$ and $w_{(h_m)}$(recall that we assume $k\ge2$), we get a point which, by Lemma \ref{limit_},  does not belong to $\overline{f_\sigma^{-1}(y)}$. Since $m$ can be arbitrarily large, it follows that $\overline{f_\sigma^{-1}(y)}$ cannot contain any interval, so that it is nowhere dense and we can conclude.

\noindent Suppose now that $y\in\mathcal{F}^\sigma\setminus\mathcal{Q}_2$. Then in the last member of \eqref{0a} we can choose a sequence $\{h_i\}$ coinciding with $\{m_i\}$ up to an arbitrarily large index $\bar{i}$ to define a new point in $f^{-1}_\sigma(y)$ which will be arbitrarily close to $x$.  This shows that $f_\sigma^{-1}(y)$ is dense in itself, and since the points in $\mathcal{B}_1^\sigma \cup \mathcal{B}_2^\sigma$ are limit points of $f_\sigma^{-1}(y)$ by construction,  so is $\overline{f_\sigma^{-1}(y)}$. Moreover,  it can be seen with the same argument applied above that $\overline{f_\sigma^{-1}(y)}$ does not contain any interval. 
\end{proof}
\begin{rem}
Notice that we excluded dyadic rationals in the hypotheses of Theorem \ref{cantorr} because the points belonging to the set $S_3$ described in points 2. and 4. of Lemma \ref{limit_} are in general isolated points of the fiber.
\end{rem}
\begin{thm}
There exists a measurable subset $\Gamma\subset \mathbb{I}$ such that $f_\sigma(\Gamma)$ is not measurable.
\label{nonm}
\end{thm}
\begin{proof}
Set $\mathcal{M}=\left\lbrace x\in \mathbb{I}: \widetilde{x}\ \text{does not have $w_{\epsilon}$ as a $k$-factor}\right\rbrace$.
Let $V\subset \mathbb{I}$ be a Vitali set. The optimality condition \ref{o.c} implies that $f_{\sigma}(\mathcal{M})=\mathbb{I}$.  Therefore, there is $\Gamma\subset \mathcal{M} : f_\sigma (\Gamma)=V$. Observe then that the points in $\mathcal{M}$ are not normal real numbers, in the sense (equivalent to the usual definition, see for instance \cite{niven1951definition}) that, partitioning their binary expansion in words of length $k$, the asymptotic frequency of the block $w_\epsilon$ is evidently 0. Therefore we have $m(\mathcal{M})=0$  (see for instance \cite{harman2003one}). By completeness of the Lebesgue measure, we have $m(\Gamma)=0$.
\end{proof}
\noindent It is not trivial to say for which (if any) $k$-block erasing substitutions $\sigma$ the set $f_\sigma^{-1}(V)$ can be Lebesgue-measurable and, if so, whether its measure can be strictly positive. Notice that Theorem \ref{nonm} can also be obtained as an easy consequence of a classical result by Purves (\cite{purves1966bimeasurable}), stating that, if $f$ is a bi-measurable map from a standard Borel space $X$ to a Polish
space $Y$ , then the cardinality of points having an uncountable $f$-preimage is at most countable. We preferred to provide a (simple) direct proof mainly to make clearer the meaning of the questions above.

\section{Dynamics of $f_\sigma$ based on the erasing class}
\label{dynamicss}
\noindent In this section we will study the connection between the erasing class of $\sigma$ and the dynamical properties of $f_\sigma$. \textbf{Throughout the Section, $\sigma$ is a $k$-block erasing substitution defined by Eq.\eqref{sigma}.} 

\noindent We start ``getting rid'' of the extreme case of boundedly erasing substitutions in the next Lemma.
\begin{lem}
If $\sigma$ is boundedly erasing, then $\sigma$ does not verify the optimality condition \eqref{o.c}. Moreover, there exists $n\in\mathbb{N}$ such that, for every $x\in\mathbb{I}$ such that its $f_\sigma$-orbit has empty intersection with $\mathcal{Q}_2$, we have $f^n_\sigma(x)=0$.
\label{boundedly}
\end{lem}
\begin{proof}
Since $\sigma$ is boundedly erasing, there exists $n=\max_{w\in\{0,1\}^*}\epsilon(w)$ (as $\epsilon(\cdot)\ge 0$ is integer and bounded by assumption). Then, for every $u\in\{0,1\}^\omega$ and for every $m>0$, $\sigma^n(u_1\dots u_m)=\epsilon$, whence, by induction on $m$, we have 
\begin{equation}
\forall u\in\{0,1\}^\omega,\ \ \sigma^n(u)=\epsilon.
\label{boundedd}
\end{equation} 
Take $w^{(0)}\in\{0,1\}^{\omega}$. Suppose by absurde that $\sigma$ verifies the optimality condition. Then there exists $w^{(1)}\in\{0,1\}^\omega$ such that $\sigma(w^{(1)})=w^{(0)}$. Iterating the argument, it  follows that we can construct a sequence of nonempty infinite words $\{w^{(m)}\}_{m\in\mathbb{N}}$ such that, for every $m\in\mathbb{N}$, $\sigma(w^{(m)})=w^{(m-1)}$. But then the word $w^{(N)}$ contradicts \eqref{boundedd} for $N$ large enough. Take now $x\in\mathbb{I}$ such that its $f_\sigma$-orbit is does not contain dyadic rationals other than 0 and 1. Suppose that $f_\sigma^j(x)\ne 0$ for every $j\in\{1,\dots,n-1\}$. Then $\sigma^{n-1}(\widetilde{x})=w_{\epsilon}^{\infty}$ which implies $f_\sigma^{n}(x)=0$. Recalling that 0 is a fixed point of $f_\sigma$ by definition, we have our thesis.
\end{proof}
\noindent In the following Lemma we prove a technical result which will be widely used in the following.

\begin{lem} \label{lem_strongly}
If $\sigma$ is strongly erasing and verifies the optimality condition \ref{o.c}, then, for every $w\in\{0,1\}^*$, $u\in\{0,1\}^\infty$, there is a positive integer $h$ depending only on $w$ and a word $v\in\{0,1\}^\infty$, such that: 
\begin{itemize}
\item [a)] $\sigma^h(wv)$ has $u$ as a prefix (in particular it coincides with $u$ if $|u|=\infty$);
\item [b)] for every finite integer sequence $\{n_0,\dots,n_{h-1}\}\in (\mathbb{N}_0)^{h}$, the word $\sigma_w^{i}(v)$ has $e_{(i)}w_\epsilon^{n_i}$ as a prefix for $i\in\{0,\dots,h-1\}$, where $\{e_{(i)}\}_{i=0,\dots,h-1}$ is a set of erasing $k$-extensions of $w$.

\noindent Moreover:
\item [c)] if $\sigma$ is completely erasing, we have the same results and we can choose $h=\epsilon(w)$.
\end{itemize}
\end{lem}

\begin{proof}
We start fixing $w\in \{0,1\}^*$ and $u\in \{0,1\}^\infty$.  
\begin{enumerate}
\item $|u|<\infty$. 

\noindent It follows from the optimality condition \ref{o.c} that, for every finite sequence of non-negative integers $\{n_0,\dots,n_{h-1}\}\in(\mathbb{N}_0)^{h}$, we can find $h$ finite binary words $p_{(0)},p_{(1)},\ldots,p_{(h-1)}$ such that,  for suitable $q_{(i)}\in \{0,1\}^*$ ($i=0,\ldots,h-1$), we have $$\sigma(p_{(0)})=uq_{(0)}\ , \ \sigma(p_{(i)})=e_{(h-i)}w_\epsilon^{n_{h-i}}p_{(i-1)}q_{(i)} \quad \forall i=1,\ldots,h-1.$$ It follows that the integer $h$ and the word $v=e_{(0)}w_\epsilon^{n_0}p_{(h-1)}$ verify the claim, because $\sigma^h(wv)$ has $u q_{(0)}$ as a prefix and, for every $i\in\{0,\dots, h-1\}$, $\sigma^{i}_w(v)$ has $e_{(i)}w_\epsilon^{n_i}$ as a prefix. Notice that this occurrence of $w_\epsilon^{n_i}$ will be a $k$-factor in $\sigma^{(i)}(wv)$ by construction.

\item $|u|=\infty$. 

\noindent It follows from the optimality condition \ref{o.c}, and in particular from Remark \ref{optimfinite}, that, for every finite sequence of non-negative integers $$\{n_0,\dots,n_{h-1}\}\in(\mathbb{N}_0)^{h},$$ we can find $h$ finite binary words $p_{(0)},p_{(1)},\ldots,p_{(h-1)}$ such that $\sigma(p_{(0)})=u$
and $\sigma(p_{(i)})=e_{(h-i)}w_\epsilon^{n_{h-i}}p_{(i-1)}$ for $i=1,\ldots,h-1$. It follows that the integer $h$ and the word $v=e_{(0)}p_{(h-1)}$ verify the claim.
\end{enumerate}
\noindent Finally, recalling that completely erasing substitutions  are always strongly erasing, a) and b) are obviously true also for them, and since $\sigma$ in this case is alternating, we can take as erasing $k$-extension $e_{(i)}$ the word $$(w_\epsilon)_{k-|e_{(i)}|+1}\dots (w_\epsilon)_{k},$$ which readily implies that $h$ can be taken equal to $\epsilon(w)$, proving point c).
\end{proof}

\begin{lem}
If $\sigma$ is strongly erasing and verifies the optimality condition \ref{o.c}, then there is an uncountable set $\mathcal{D}\subset\mathbb{I}$  such that the $f_\sigma$-orbit of every $x\in\mathcal{D}$ is dense in $\mathbb{I}$. Moreover, $\mathcal{D}$ is dense in $\mathbb{I}$.
\label{_3_}
\end{lem}
\begin{proof}
Let us indicate by $\{w^{(n)}\}_{n\in\mathbb{N}_0}$ an enumeration of all finite binary words. We will construct inductively a point $x=0.\widetilde{x}$ with the property that there exists a sequence of integers $\{h_n\}_{n\in \mathbb{N}_0}$ such that  $\sigma^{h_n}(\widetilde{x})$ has $w^{(n)}$ as a prefix for every $n$.  

\noindent By Lemma \ref{lem_strongly}, there exists $u^{(1)}\in\{0,1\}^+$ and $h_1\in\mathbb{N}$ such that $\sigma^{h_1}(w^{(0)}u^{(1)})$ has $w^{(1)}$ as a prefix.  Let $\{e_{(j)}^{(0)}\}_{j\in\{0,\dots,h_{1}-1\}}$ be the erasing $k$-extensions of $w^{(0)}$. By Lemma \ref{lem_strongly}, for $\{n^1_0,\dots,n^1_{h_1-1}\}\in (\mathbb{N}_0)^h$, we can select $u^{(1)}$ so that the word $\sigma_{w^{(0)}}^j\left(u^{(1)}\right)$ has $e_{(j)}^{(0)}w_\epsilon^{n_j^1}$ as a prefix for $j\in\{0,\dots,h_1-1\}$. 

\noindent Set $p^{(1)}=w^{(0)}u^{(1)}$. Suppose that we arrived at the step $i-1$ and we have defined $p^{(i-1)}=w^{(0)}u^{(1)}\dots u^{(i-1)}$. At the $i$-th step we have that, again by Lemma \ref{lem_strongly}, there exists $u^{(i)}\in \{0,1\}^+$ and $h_i\in\mathbb{N}$ such that $\sigma^{h_i}(p^{(i-1)}u^{(i)})$ has $w^{(i)}$ as a prefix. Let $\{e_{(j)}^{(i-1)}\}_{i\in\{0,\dots,h_{1}-1\}}$ be the erasing $k$-extensions of $w^{(i-1)}$. By Lemma \ref{lem_strongly}, for $\{n^i_0,\dots,n^i_{h_i-1}\}\in (\mathbb{N}_0)^h$, we can select $u^{(i)}$ so that the word $\sigma_{p^{(i)}}^j\left(u^{(i)}\right)$ has $e_{(j)}^{(i-1)}w_\epsilon^{n^i_j}$ as a prefix for $j\in\{0,\dots,h_i-1\}$. 
By the construction made in Lemma \ref{lem_strongly}, we have $\sigma^{h_i}(p^{(i)})=\epsilon$, and since $p^{(i-1)}$ is always a prefix of $p^{(i)}$, it follows that $h_i\ge h_{i-1}$ for every $i$. We can then define the word $v$ as 
\begin{equation}
v=w^{(0)}\prod_{j=1}^\infty u^{(j)}
\label{d_orbit}
\end{equation} 
so that by construction the $\sigma$-orbit of $v$ has $w^{(n)}$ as a prefix for every $n\in\mathbb{N}_0$. Notice further that different choices for $n_0^i$ ($i\in\mathbb{N}$) produce different words in Eq.\eqref{d_orbit}, which proves that we have uncountably many choices for $v$. If, for every $j\in\mathbb{N}_0$, we take the sequence $(n_j^i)_{i\in\mathbb{N}}$ in the set $\Xi^+$, we are ensured that $\sigma^j(v)$ will not end in $0^\infty$ for every $j$ (either in case $w_\epsilon=0^k$ or not). Therefore, the point $x=0.v=0.\widetilde{x}$ has a dense $f_\sigma$-orbit and the set $\mathcal{D}$ of points with a dense $f_\sigma$-orbit is uncountable.
Finally, by the arbitrariness of the prefix $w^{(0)}$, the set $\mathcal{D}$ is dense in $\mathbb{I}$.
\end{proof}
\noindent Since the sequence of words $\left(w^{(n)}\right)_{n\in\mathbb{N}_0}$ can be arbitrary, in the previous proof we also proved the following
\begin{corol}
Suppose that $\sigma$ is strongly erasing and verifies the optimality condition \ref{o.c}. Then, if $(w^{(n)})_{n\in\mathbb{N}_0}$ is any sequence of finite binary words, there is $x\in [w^{(0)}]$ and positive integers $h_1,h_2\dots$ (with $h_i$ depending only on $w^{(i)}$ for every positive integer $i$) such that $f_\sigma^{h_n}(x)\in [w^{(n)}]$ for every $n\in\mathbb{N}$. Moreover, exploiting point c) of Lemma \ref{lem_strongly}, we can write $\widetilde{x}$ as: $\widetilde{x}=\prod_{i=1}^\infty v^{(i)}$ where, for every positive integer $n$, $\sigma^{\epsilon\left({v^{(1)}\dots v^{(n-1)}}\right)}_{v^{(1)}\dots v^{(n-1)}}\left(v^{(n)}\right)$ has $w^{(n)}$ as a prefix.

\label{visitingcylinders}
\qed
\end{corol}

\begin{lem}
If $\sigma$ is strongly erasing and verifies the optimality condition \ref{o.c}, then $f_\sigma$ has $1/2$-sensitive dependence on initial conditions.
\label{_2_}
\end{lem}
\begin{proof}
Take $x\in\mathbb{I}$.  If $y\in\mathbb{I}$ is such that $\widetilde{x}$ and $\widetilde{y}$ have in common a prefix $p$ long enough, then $\left|x-y \right|<\delta$ for arbitrarily small $\delta>0$.  By Lemma \ref{lem_strongly}, there exists a positive integer $n$ (depending only on $p$) such that, for every $u\in\{0,1\}^\infty$, there is $v\in \{0,1\}^\infty$ such that $\sigma^n(pv)$ has $u$ as a prefix. Set $u=0^\infty$ if $0.\sigma^n(\widetilde{x})>1/2$ and $u=1^\infty$ if $0.\sigma^n(\widetilde{x})\le 1/2$. Then,  if we take $z\in \mathbb{I}$ such that $\widetilde{z}$ has $pv$ as a prefix, we have
$|f_\sigma^n(x)-f_\sigma^n(z)|\ge 1/2$.
The thesis follows from the arbitrariness of $\delta$ and of the point $x\in \mathbb{I}$.
\end{proof}

\noindent A particular case of Corollary \ref{visitingcylinders} is that it is possible to obtain a point $x$ whose orbit visits infinitely many times any cylinder $[w]$ by simply imposing that $\{w^{(n)}\}$ in the proof of Lemma \ref{_3_} is the constant sequence: $w^{(n)}\equiv w\ \forall n\in\mathbb{N}$. Notice however that, assuming $\sigma$ strongly erasing, this does not provide in general a \textit{periodic} point, but only a (uniformly) recurrent point. Indeed, for every prefix $w\in\{0,1\}^*$, the fact that $\sigma$ is strongly erasing implies that we can find a positive integer $n$ and a word $u\in\{0,1\}^*$ such that $\sigma^n(wu)=wp$ for some word $p\in\{0,1\}^{<k}$.  However, to get started with the construction of a periodic point this is not enough, because we should also have that $p$ is a prefix of $u$. This is in general not the case, and therefore to ensure the existence of periodic points we should assume the stronger hypothesis that $\sigma$ is completely erasing. As we will see in Lemma \ref{_1_}, this in fact implies a lot more than the existence of \textit{one} periodic point, but first we have to refine the technique employed in the proof of Lemma \ref{lem_strongly} to get a further result. 

\noindent To lighten the notation, in the following of the paper we set $\epsilon_i=(w_\epsilon^\infty)_i$ for every $i\in \mathbb{N}$.

\begin{lem}
If $\sigma$ is completely erasing and verifies the optimality condition \ref{o.c}, then for every $w\in\{0,1\}^*$ such that $\sigma^{h}(w)=p\in\{0,1\}^*$ and every $u\in\{0,1\}^\omega$, there is $v\in\{0,1\}^\omega$ such that $\sigma^h(wv)=pu$.
\label{lem_completely}
\end{lem}
\begin{proof}
\noindent Let be $u\in\{0,1\}^\omega$ and $w\in\{0,1\}^*$ such that $\sigma^h(w)=p\in \{0,1\}^*$. Set $w^{(j)}=\sigma^j(w)$, $n_jk+m_j=|w^{(j)}|$ for $0\le j\le h$, $n_j\in\mathbb{N}_0$, $0\le m_j<k$ (notice that $w^{(0)}=w$). By the optimality condition, there is $q^{(1)}\in\{0,1\}^\omega$ such that $\sigma\left(q^{(1)}\right)=u$. 
Therefore, $\sigma(aq^{(1)})=\sigma(a)u$ for every $a\in\{0,1\}^{nk}$, $n\in\mathbb{N}_0$. Set $v^{(1)}=\epsilon_{m_{h-1}+1}\dots\epsilon_{k}q^{(1)}$, so that, recalling Eq.\eqref{alterneras}, $\sigma(w^{(h-1)}v^{(1)})=w^{(h)}u=pu$. 

\noindent Iterating the argument, we get the existence of words $v^{(j)}$ ($1\le j\le h$) such that $\sigma(w^{(h-j)}v^{(j)})=w^{(h-j+1)}{v^{(j-1)}}$, so that $\sigma^h(wv^{(h)})=pu$, and therefore $v=v^{(h)}$ verifies the claim. 
\end{proof}
\begin{lem}
If $\sigma$ is completely erasing and verifies the optimality condition \ref{o.c}, then the set of $f_\sigma$-periodic points of period $p$, indicated by $\mathcal{P}_p$, is uncountable for every $p\in\mathbb{N}$. Moreover, $\mathcal{P}=\cup_{p\in\mathbb{N}}\mathcal{P}_p$ is dense in $\mathbb{I}$.
\label{_1_}
\end{lem}
\begin{proof}
Take $u^{(0)}\in\{0,1\}^{\ge k}$. Since $\sigma$ is alternating, we can apply it to finite words of any length (without dropping any digit for words of length not multiple of $k$). We will show an iterative construction leading to a periodic point belonging to $[u^{(0)}]$. 

\noindent Set $|u^{(0)}|=n_0k+m_0$ with $n_0\in\mathbb{N}$ and $0\le m_0 <k $. By Lemma \ref{lem_strongly}, there exists a binary word $u\in\{0,1\}^\omega$ such that $\sigma^{\epsilon(u^{(0)})}\left(u^{(0)}u\right)=u^{(0)}\prod_{j=m_0+1}^{\infty}\epsilon_{j}$, so that there must be $p_0,s_0\in\mathbb{N}_0$ such that 
$$\sigma^{\epsilon(u^{(0)})}\left(u^{(0)}u_1\dots u_{p_0}\right)=u^{(0)}\epsilon_{m_0+1}\epsilon_{m_0+2}\dots\epsilon_{m_0+s_0}.$$ 
Set then $s_0=r_0k+q_0$ for $r_0\in\mathbb{N}_0$, $0\le q_0<k$ and define
$$w^{(0)}=u^{(0)}\epsilon_{m_0+1}\epsilon_{m_0+2}\dots\epsilon_{m_0+s_0}\ , \ u^{(1)}=\left(\prod_{j=1}^{k-q_0} \epsilon_{m_0+s_0+j}\right)u_1\dots u_{p_0},$$
where the product can be replaced by the empty word if $q_0=0$. Recalling Eq.\eqref{alterneras}, we have $\sigma(u^{(0)})=\sigma(w^{(0)})$, so that $\epsilon(u^{(0)})=\epsilon(w^{(0)})$. Since by construction $s_0+k-q_0$ is a multiple of $k$, we have thus $\sigma\left(w^{(0)}u^{(1)}\right)=\sigma \left(u^{(0)}u_1\dots u_{p_0}\right)$ and therefore $\sigma^{\epsilon(w^{(0)})}(w^{(0)}u^{(1)})=w^{(0)}$. 

\noindent Suppose now that there are $i+2$ finite binary words $w^{(0)},\dots,w^{(i)},u^{(i+1)}$ such that $\sigma^{\epsilon(w^{(0)})}(w^{(0)}\dots w^{(i)}u^{(i+1)})=w^{(0)}\dots w^{(i)}$. Set $|w^{(0)}\dots w^{(i)}u^{(i+1)}|=n_{i+1}k+m_{i+1}$ with $n_{i+1}\in\mathbb{N}$ and $0\le m_{i+1} <k $. By Lemma \ref{lem_completely}, there exists a binary word $u\in\{0,1\}^\omega$ such that $\sigma^{\epsilon(w^{(0)})}\left(w^{(0)}\dots w^{(i)}u^{(i+1)}u\right)=w^{(0)}\dots w^{(i)}u^{(i+1)}\prod_{j=m_{i+1}+1}^{\infty}\epsilon_{j}$, so that there must be $p_{i+1},s_{i+1}\in\mathbb{N}_0$ such that 
$$\sigma^{\epsilon(w^{(0)})}\left(w^{(0)}\dots w^{(i)}u^{(i+1)}u_{1}\dots u_{p_{i+1}}\right)=$$ $$w^{(0)}\dots w^{(i)}u^{(i+1)}\epsilon_{m_{i+1}+1}\epsilon_{m_{i+1}+2}\dots\epsilon_{m_{i+1}+s_{i+1}}.$$ 
Set then $s_{i+1}=r_{i+1}k+q_{i+1}$ for $r_{i+1}\in\mathbb{N}_0$, $0\le q_{i+1}<k$ and define
$$w^{(i+1)}=u^{(i+1)}\epsilon_{m_{i+1}+1}\epsilon_{m_{i+1}+2}\dots\epsilon_{m_{i+1}+s_{i+1}}$$ 
$$u^{(i+2)}=\left(\prod_{j=1}^{k-q_{i+1}} \epsilon_{m_{i+1}+s_{i+1}+j}\right)u_1\dots u_{p_{i+1}}$$ 
where the product can be replaced by the empty word if $q_{i+1}=0$. Since by construction $s_{i+1}+k-q_{i+1}$ is a multiple of $k$, we have thus $$\sigma\left(w^{(0)}\dots w^{({i})}w^{(i+1)}u^{(i+2)}\right)=\sigma \left(w^{(0)}\dots w^{(i)}u^{(i+1)}u_1\dots u_{p_{i+1}}\right),$$ and therefore $\sigma^{\epsilon(w^{(0)})}(w^{(0)}\dots w^{(i+1)}u^{(i+2)})=w^{(0)}\dots w^{(i+1)}$. 

\noindent By induction on $i$ we get the existence of a countable family of words verifying, for every $i\in\mathbb{N}_0$, $$\sigma^{\epsilon(w^{(0)})}\left(w^{(0)}\dots w^{(i+1)}\right)=w^{(0)}\dots w^{(i)}.$$
It follows that the word $w=\prod_{i=0}^\infty w^{(i)}$ is $\sigma$-periodic with period $\epsilon(w^{(0)})$. As observed multiple times, in the applications of Lemma \ref{lem_strongly} (as well as of Lemma \ref{lem_completely}) we can ensure that, for every $h\in\mathbb{N}$, $\sigma^h(w)$ does not end in $0^\infty$, so it follows that the point  $x=0.\prod_{i=0}^\infty w^{(i)}$ is an $f_\sigma$-periodic point of period $\epsilon(w^{(0)})$.

\noindent The density of the set $\mathcal{P}$ of $f_\sigma$-periodic points follows from the arbitrariness of the word $u^{(0)}$. The uncountability of $\mathcal{P}_p$ ($p\in\mathbb{N}$) follow from the same argument used in Lemma \ref{_3_}.
\end{proof}
\begin{rem}\label{fixedd}
Notice that, in the previous proof, taking $u^{(0)}=w_\epsilon^m$ ($m\in\mathbb{N}$) we obtain fixed points for $f_\sigma$, and that for every $m$ there are uncountably many of them. 
\end{rem}

\noindent In the proof of Lemma \ref{lem_completely} we actually also proved the following
\begin{corol}
If $\sigma$ is completely erasing and verifies the optimality condition \ref{o.c}, then for every $w,u\in\{0,1\}^*$ there is $s\in\mathbb{N}_0$ and $v\in\{0,1\}^*$ such that, setting $|w|=nk+m$ ($n\in\mathbb{N}_0\, , 0\le m<k$), $\sigma^{\epsilon(w)}(wv)=u\prod_{i=m+1}^{s}(w_\epsilon^\infty)_i$. Notice that by construction $\sigma_{w}(\prod_{i=m+1}^{s}(w_\epsilon^\infty)_i)=\epsilon$. 
\label{fillepsilon}
\qed
\end{corol}
\noindent This in turn implies the following
\begin{lem}
If $\sigma$ is completely erasing and verifies the optimality condition \ref{o.c}, then, for every $n\in\mathbb{N}$, there is a word $w\in\{0,1\}^*$ such that $\epsilon(w)=n$.
\label{order_}
\end{lem}
\begin{proof}
By Corollary \ref{fillepsilon}, there is a word $w^{(1)}$ such that $\sigma(w^{(1)})=w_\epsilon q^{(1)}$, with $\sigma_{w_\epsilon}(q^{(1)})=\sigma(q^{(1)})=\epsilon$. Applying further $n-1$ times the same argument one gets the existence of $w^{(2)},\dots,w^{(n)}$ such that $\sigma(w^{(i+1)})=w^{(i)}q^{(i)}$ with $\sigma_{w^{(i)}}(q^{(i)})=\epsilon$ for $1\le i\le n-1$, so that $\sigma^n(w^{(n-1)})=\epsilon$ and therefore $\epsilon(w^{(n-1)})=n$.
\end{proof}

\noindent Lemmas \ref{_3_}, \ref{_2_} and \ref{_1_} imply the following
\begin{thm}
If $\sigma$ is completely erasing and verifies the optimality condition \ref{o.c}, then $f_\sigma$ is Devaney chaotic. \qed
\end{thm} 
\begin{rem} It is well-known that, in the definition of Devaney chaos, assuming sensitive dependence on initial conditions is redundant for continuous maps defined on (infinite) metric spaces (\cite{banks1992devaney}), and on the interval the existence of one dense orbit implies both sensitive dependence and the existence of a dense set of periodic points (\cite{vellekoop1994intervals}). Because our general object has dense discontinuities, we proved explicitly the three properties. 
\end{rem}

\begin{thm}
If $\sigma$ is strongly erasing and verifies the optimality condition \ref{o.c}, then for every $w\in\{0,1\}^*$, there is $h\in\mathbb{N}$ such that $f_\sigma^{h}([w])=\mathbb{I}$, and therefore $f_\sigma$ is topologically mixing.
\label{mixing}
\end{thm}

\begin{proof}
By Lemma \ref{lem_strongly}, for every $w\in\{0,1\}^*$ there is $h\in\mathbb{N}$ such that, for every $y\in \mathbb{I}$ there is $v\in\{0,1\}^\omega$ such that $y=0.\sigma^{h}(wv)$, where we can take $\sigma^i_w(v)$ not ultimately constantly 0 for every $i\in\{0,\dots,h-1\}$ by exploiting point b) of Lemma \ref{lem_strongly} to insert, if necessary, the words $w_\epsilon$ as $k$-factors. Therefore, $f_\sigma^{h}(x)=y$ if $wv=\widetilde{x}$, which implies that $f_\sigma^{h}([w])=\mathbb{I}$. Take now $A,B$ open subsets of $\mathbb{I}$. Since open sets in the interval are countable union of (disjoint) open intervals, there is $w\in\{0,1\}^+$ such that $[w]\subset A$. If $h\in\mathbb{N}$ is that of the previous argument,  by Lemma \ref{surj}, $f_\sigma^n(A)\cap B\ne \emptyset$ for every $n\ge h$, which means that $f_\sigma$ is topologically mixing.  
\end{proof}

\noindent For continuous interval maps, Devaney chaos implies Li-Yorke chaos (\cite{huang2002devaney}), and the existence of a Li-Yorke pair implies the existence of an uncountable scrambled set (\cite{kuchta1989two}). In our case, since the map $f_\sigma$ is discontinuous on a countable subset of $\mathbb{I}$, Li-Yorke chaos is a priori not given, and moreover the existence of just one Li-Yorke pair does not guarantee Li-Yorke chaos. However, assuming that $\sigma$ is completely erasing, we can in fact prove that $f_\sigma$ exhibits Li-Yorke chaos. 
\begin{thm}
\noindent If $\sigma$ is completely erasing and verifies the optimality condition \ref{o.c}, then $f_\sigma$ is Li-Yorke chaotic, meaning that there exists an uncountable \textit{scrambled} subset of its domain, that is a set $S\subset \mathbb{I}$ such that, for every $x,y\in S$:
\begin{equation}
\liminf_{n\to\infty}|f_\sigma^n(x)-f_\sigma^n(y)|=0\quad ,\quad \limsup_{n\to\infty}|f_\sigma^n(x)-f_\sigma^n(y)|>0
\label{LYC}
\end{equation}
\label{LiYorke}
\end{thm}
\begin{proof}
\noindent Consider a sequence $W=\{w^{(n)}\}_{n\in\mathbb{N}}$ of finite binary words such that $\{|w^{(n)}|\}_{n\in\mathbb{N}}$ is an unbounded sequence of positive integers. By Corollary \ref{visitingcylinders}, we can find $x\in\mathbb{I}$ such that its $f_\sigma$-orbit visits $[w^{(n)}]$ for every $n\in\mathbb{N}$,  and since $\sigma$ is completely erasing,  we can write $\widetilde{x}$ as: $\widetilde{x}=\prod_{i=1}^\infty v^{(i)}$ where, for every positive integer $n$ there is $q^{(n)}\in\{0,1\}^*$ such that 
\begin{equation}
\sigma^{\epsilon\left({v^{(1)}\dots v^{(n-1)}}\right)}_{v^{(1)}\dots v^{(n-1)}}\left( v^{(n)} \right)=w^{(n)}q^{(n)}.
\label{vvvvv}
\end{equation}
Moreover, by Corollary \ref{fillepsilon}, we can choose $q^{(n)}$ such that $\epsilon(w^{(n)}q^{(n)})=\epsilon(w^{(n)})$. 

\noindent Notice that \eqref{vvvvv} is the only property used to define $v^{(n)}$, which means that in fact we can define a family of points $X\subset\mathbb{I}$ as follows:
\begin{equation}
X=\{x=0.\prod_{n=1}^\infty v^{(n)}\ \text{where $v^{(n)}$ is such that \eqref{vvvvv} holds for every}\ n\in\mathbb{N}_0\}
\end{equation} 

\noindent In particular, by Lemma \ref{lem_completely}, we are free to choose $v^{(n)}$ so as to have that
\begin{equation}
\sigma^{\epsilon\left({v^{(1)}\dots v^{(n-1)}}\right)-1}_{v^{(1)}\dots v^{(n-1)}}\left( v^{(n)}\right) \text{has }u^{(n)} \text{ as a prefix},
\label{choice1}
\end{equation}
where $\sigma(u^{(n)})=w^{(n)}q^{(n)}$ and $u^{(n)}$ does not have $w_{\epsilon}$ as a $k$-factor (notice that we do not have to insert a suffix after $q^{(n)}$ because we know that Eq.\eqref{vvvvv} holds).

\noindent On the other hand, we can also make the choice
\begin{equation}
\sigma^{\epsilon\left({v^{(1)}\dots v^{(n-1)}}\right)-1}_{v^{(1)}\dots v^{(n-1)}}\left(v^{(n)}\right) \text{ has }w_{\epsilon}^2u^{(n)} \text{ as a prefix}.
\label{choice2}
\end{equation}
We underline that, for every $n\ge 1$, we can make choice $\eqref{choice1}$ or choice \eqref{choice2} while still retaining the property \eqref{vvvvv}. We want to use this to construct an uncountable scrambled subset of $\mathbb{I}$. For $\alpha=\alpha_{(1)}\alpha_{(2)}\dots \in\Omega$ (see Definition \ref{defiomegaxi}), let us define the set $X_{\alpha}\subset X$ of points in $\mathbb{I}$ such that, for every $x\in X_\alpha$,
\begin{equation}
\widetilde{x}=\prod_{n=1}^\infty v^{(n)}\ \text{where, for every $n$}:
\begin{cases}
\text{\eqref{choice1} holds if $\alpha_{(n)}=1$}\\
\text{\eqref{choice2} holds if $\alpha_{(n)}=0$}. 
\end{cases} 
\label{xxxxxx}
\end{equation}
Pick now exactly one point $x_\alpha$ from every set $X_\alpha$ and set $X_{\Omega}=\bigcup_{\alpha\in\Omega}\{x_{\alpha}\}$. 
We will show now that $X_\Omega$ is an uncountable scrambled set. Indeed, let $\alpha\ne\beta$ be two elements of $\Omega$. Since both $\widetilde{x_\alpha}$ and $\widetilde{x_\beta}$ verify \eqref{vvvvv} for every $n$, we have that $f_\sigma^n(x_\alpha),f_\sigma^n(x_\beta)\in [w^{(n)}]$. Since $|w^{(n)}|$ is unbounded, it follows that 
$$\liminf_{n\to\infty}|f_\sigma^n(x_\alpha)-f_\sigma^n(x_\beta)|=0.$$
Moreover, only for finitely many integers $n$ the word $u^{(n)}$ is shorter than 2$k$, so that $w_\epsilon$ is different from both $u^{(n)}_1\dots u^{(n)}_k$ and $u^{(n)}_{k+1}\dots u^{(n)}_{2k}$ for large enough $n$. It follows from Eqs. \eqref{choice2} and \eqref{xxxxxx} that, if $n$ is such that $\alpha_{(n)}\ne\beta_{(n)}$, then $q=\epsilon(v^{(1)}\dots v^{(n-1)})-1$ implies $|f^q_\sigma(x_\alpha)-f_\sigma^q(x_\beta)|> 2^{-2k}$. Since $\alpha$ and $\beta$ are not ultimately coinciding, it follows that $\limsup_{n\to\infty}|f_\sigma^n(x_\alpha)-f_\sigma^n(x_\beta)|>0$.
Finally, $X_\Omega$ is uncountable as $\alpha\ne\beta$ implies that $x_\alpha\ne x_\beta$ because by construction they have distinct $f_\sigma$-orbits.  
\end{proof}

\noindent Let us now devote our attention to the topological entropy of the maps $f_\sigma$. The investigation of topological entropy for discontinuous maps has become more intense in last years, and in particular maps with dense discontinuities have been studied also from this point of view (see for instance \cite{korczak2015topological}, where Darboux, Baire-1 maps are considered). In the following result we will see that the topological entropy of $f_\sigma$ depends on how $\epsilon(w)$ goes asymptotically with $|w|$. The proof technique, which somewhat recalls the classical construction of horseshoes developed originally for continuous interval maps (\cite{misiurewicz2010horseshoes}), is based on the idea of ``spreading'' a set of points belonging to each interval from a suitable partition of $\mathbb{I}$ over all the other intervals of the partition.

\noindent We start with our last technical notational tool.
\begin{defn}
\noindent For every $n\in\mathbb{N}$ and every $v\in \{0,1\}^n$, we indicate by $v\overline{w_\epsilon}$ the word $v\epsilon_{n+1}\dots\epsilon_{n+k}$.
\end{defn}
\noindent Clearly, if $\sigma$ is completely erasing, for any finite word $v$ we have $\epsilon(v)=\epsilon(v\overline{w_\epsilon})$.
\begin{thm}
Suppose that $\sigma$ is a completely erasing substitution verifying the optimality condition \ref{o.c}.  Then $f_\sigma$ has infinite topological entropy if, for every $w\in\{0,1\}^*$, 
\begin{equation}
\label{limepsilon}
\lim_{|w|\to\infty} \frac{|w|}{\epsilon(w)}=\infty.
\end{equation}
\label{entrop}
\end{thm}
\begin{proof}
\noindent For every $w\in\{0,1\}^+$, set $F({\left| w \right|})=\max_{u\in\{0,1\}^{|w|}} \varepsilon(u) $. For $x_1,x_2 \in \mathbb{I}$ let us introduce the metric:
\begin{equation} \label{N_metric}
d_n(x_1,x_2) \coloneqq \max\{\left| f_\sigma^i(x_1)-f_\sigma^i(x_2)\right| : 0\le i \le n\}
\end{equation}
We say that a subset $S\subset \mathbb{I}$ is $(n,\varepsilon)$-separated in the metric $d_n$ if for all $x_1,x_2\in S$,  $x_1 \neq x_2$ we have that $d_N(x_1,x_2)\ge \varepsilon$ and we indicate by $\left| (n,\epsilon) \right|$ the maximum cardinality of an $(n,\varepsilon)$-separeted set.  We recall that the topological entropy $h$ of the map $f_\sigma$ can be written as (\cite{dinaburg1970correlation,bowen1971entropy}):
\begin{equation*}
h(f_\sigma)=\lim_{\varepsilon \to 0} \left( \limsup_{n \to \infty} \frac{1}{n} \log \left| (n,\varepsilon) \right| \right)
\end{equation*}

\noindent Fix now $\varepsilon=2^{-(k+1)}$. 
\begin{enumerate}
\item[\textit{Step 1.}] If $f_\sigma^t(x_1)\in [w_{(i)}\overline{w_\epsilon}]$ and $f_\sigma^t(x_2)\in [w_{(r)}\overline{w_\epsilon}]$,  for some $0\le t\le N$ ($i\neq r$), then $d_N(x_1,x_2)>\epsilon$. Indeed, we have the smallest possible difference between $f_\sigma^t(x_1)$ and $f_\sigma^t(x_2)$ when $w_{(i)}$ and $w_{(r)}$ differ at the $k$-th digit and $f_\sigma^t(x_1)=0.w_{(i)}\overline{w_\epsilon} 0^\infty\ ,\ f_\sigma^t(x_2)=0.w_{(r)}\overline{w_\epsilon} 1^\infty$ (or vice-versa), so that we have $\left|f_\sigma^t(x_1)-f_\sigma^t(x_2) \right|\ge\left| \frac{1}{2^k}-\frac{1}{2^{2k}} \right|\ge \frac{1}{2\cdot 2^k}=\epsilon$.

\noindent Set $t=F(k)$.

\item[\textit{Step 2.}]
For any $i,r\in \{1,\ldots, 2^k\}$, there exists $x\in [w_{(i)}\overline{w_\epsilon}]$ such that $f_\sigma^t(x)\in [w_{(r)}\overline{w_\epsilon} q^{(r)}]$ with $\sigma_{w_{(r)}\overline{w_\epsilon}}q^{(r)}=\epsilon$.

\noindent Indeed, let us fix $i,r\in \{1,\ldots, 2^k\}$.  Since by assumption $\epsilon(w) \le F(|w|)$ for every $w\in\{0,1\}^*$, we have $\epsilon(w_{(i)})=\epsilon(w_{(i)}\overline{w_{\epsilon}})\le t$.  By Lemma \ref{order_}, there exists a word $p\in \{0,1\}^*$ such that $\epsilon(p)=t-\epsilon(w_{(i)}\overline{w_\epsilon})$, and by Corollary \ref{fillepsilon}, there is $q\in\{0,1\}^*$ such that $\sigma^{\epsilon(w_{(i)}\overline{w_\epsilon})}(w_{(i)}\overline{w_\epsilon} q)=pp'$, with $p'\in\{0,1\}^*$ such that $\epsilon(pp')=\epsilon(p)$, and therefore $\epsilon(w_{(i)}\overline{w_\epsilon} q)=t$. Then the thesis follows by applying Lemma \ref{lem_strongly}, where $w=w_{(i)}\overline{w_\epsilon} q$, $u=w_{(r)}\overline{w_\epsilon}$ and $h=t$.

\item[\textit{Step 3.}] We will define now a family of sets $\{S^i \subset [w_{(i)}\overline{w_\epsilon}]\}_{i=1,\dots,2^k}$, each containing $2^{nk}$ ($n\in \mathbb{N}$) distinct points of $\mathbb{I}$, such that $\cup_{1\le i\le 2^k}S^i$ will be an $(n,\varepsilon)$-separated set.  We will indicate by $x_l^i$ ($l\in \{1,\ldots, 2^{nk}\}$) the points of $S^i$, where the subscript $l$ identifies their standard order as real numbers, and we will construct them iteratively, requesting properties which will be ensured by increasingly long prefixes of their binary expansion. By Step 2.,  for $i_1 \in\{1,\ldots, 2^k\}$ the points $x_l^{i_1}$ can be taken such that:
\begin{equation} \label{step_1}
f_\sigma^t(x_l^{i_1})\in S^{i_2} \quad \text{for } l\in \{(i_2-1)2^{(n-1)k}+1,\ldots, i_22^{(n-1)k}\}
\end{equation}
where $i_2\in \{1,\ldots, 2^k\}$.  Hence we can define the set 
$$S^{i_1i_2}=\{x:x\in S^{i_1}, f_\sigma^t(x)\in S^{i_2}\}$$
and we have $\left| S^{i_1i_2} \right|=2^{(n-1)k}$. Moreover, notice that after further $t$ iterates of $f_\sigma$, the prefix $w_{i_2}\overline{w_\epsilon}$ of every point $f_\sigma^t(x_l^{i_1})$ is mapped again in the empty word.  Therefore for every $i_1,i_2\in \{1,\ldots, 2^k\}$ and for $l\in \{(i_2-1)2^{(n-1)k}+1,\ldots, i_22^{(n-1)k}\} $,  Step 2. allows us to have:
\begin{equation} \label{step_2}
f_\sigma^{2t}(x_l^{i_1i_2})\in S^{i_3} \quad \text{for } l\in \{(i_3-1)2^{(n-2)k}+1,\ldots, i_32^{(n-2)k}\}
\end{equation}
where $\{x^{i_1i_2}_l \}_{l=1,\ldots, 2^{(n-1)k}}\in S^{i_1 i_2}$, the subscript $l$ indicates their standard order as reals and $i_3 \in \{1,\ldots,2^k \}$.  Hence we can define the set 
$$S^{i_1 i_2 i_3}=\{x:x\in S^{i_1}, f_\sigma^t(x)\in S^{i_2}, f_\sigma^{2t}(x)\in S^{i_3}\}$$ 
and we have $\left| S^{i_1 i_2 i_3} \right|=2^{(n-2)k}$. We can proceed in this way for further $n-3$ steps in order to obtain that 
\begin{equation} \label{step_t}
f_\sigma^{jt}(x_l^{i_1\ldots i_j})\in S^{i_{j+1}} \quad \text{for } l\in \{(i_{j+1}-1)2^{(n-j)k}+1,\ldots, i_{j+1}2^{(n-j)k}\}
\end{equation}
for all $j\in \{1,\ldots,n-1\}$,  where $\{x^{i_1\ldots i_j}_l\}_{l=1,\ldots, 2^k}\in S^{i_1 \ldots i_j}$, the subscript $l$ indicates their standard order as reals and $i_{j+1}\in \{1\ldots,2^k\}$.  We finally arrive at the set 
$$S^{i_1\ldots i_{j+1}}=(\{x: x\in S^{i_1}, f_\sigma^t(x)\in S^{i_2}, \ldots, f_\sigma^{jt}(x)\in S^{i_{j+1}}\})_{j\in \{1,\ldots,n-1\}}$$
such that $\left| S^{i_1\ldots i_{j+1}} \right|=2^k$.

\item[\textit{Step 4.}] Recalling Step 1., the last argument implies that $\bigcup_{i\in \{1,\ldots, 2^k\}} S^i$ is an $(nt,\epsilon)$-separated set, so that $\left|(nt,\varepsilon) \right| \ge 2^{(n+1)k}$. From this it follows that
\begin{equation}
h(f_\sigma)\ge \lim_{k\to +\infty}\limsup_{n\to\infty} \frac{\log 2^{(n+1)k}}{(n+1)t}=\lim_{k\to +\infty} \frac{k\log 2}{F(k)}
\end{equation}
which immediately implies the thesis.
\end{enumerate}
\end{proof}
\noindent Theorem \ref{entrop} can be applied to the model case map $f_{\sigma_3}$, where $\sigma_3$ is defined in Section \ref{prelim}. Indeed, since, as observed before, all odd-indexed elements of $w$ go to $\epsilon$ in at most two iterations, the asymptotic behavior of $\epsilon(w)$ is logarithmic (for the optimal bound see \cite{dellacorte2021simplest}).

\noindent Attention has been devoted to the points around which the entropy concentrates, i.e. \textit{entropy points} in the sense of \cite{ye2007entropy}. We recall that an entropy point is a point $x$ such that the topological entropy restricted to any of its closed neighborhoods $K$ is positive, in symbols $h(f_\sigma, K)>0$. In case it always coincides with the entropy of the map on the whole space, the point is called a full entropy point. It is known that, in case of continuous maps on a compact metric space, every point is a full entropy point if the system is minimal and has positive topological entropy. When $\sigma$ is completely erasing, the system $(\mathbb{I},f_\sigma)$ is not minimal (there are periodic points). Still, we have the following
\begin{thm}
Suppose that $\sigma$ is completely erasing and verifies the optimality condition \ref{o.c}. If Eq. \eqref{limepsilon} holds, then every point $x\in\mathbb{I}$ is a full entropy point for $f_\sigma$. 
\end{thm}
\begin{proof}
\noindent Take any $x\in\mathbb{I}$. For every closed neighborhood $K$ of $x$, there is $w\in\{0,1\}^*$ such that $\widetilde{x}=wv$ for some $v\in\{0,1\}^\omega$, and that $[w]\subset K$. If $\sigma$ is completely erasing and verifies the optimality condition, by Theorem \ref{mixing} we have $f_\sigma^{\epsilon(w)}([w])=\mathbb{I}$. Therefore, there exists $S^0\subset [w]$ such that $f_\sigma^{\epsilon(w)}(S^0)=\cup_{i} S^i$ (where $S^i$ are the same as in the previous proof). Then the same argument as above implies the following bounds:
\begin{equation}
h(f_\sigma, K)\ge h(f_\sigma, [w])\ge\lim_{k\to +\infty} \frac{k\log 2}{F(k)+\epsilon(w)}=\infty
\end{equation}
\end{proof}

\noindent In \cite{pawlak2009entropy}, the author introduces a powerful concept which is useful when dealing with the dynamical properties of Darboux, Baire-1 interval maps, i.e. that of \textit{almost fixed points}. A point $x\in\mathbb{I}$ is an almost fixed point for a map $f:\mathbb{I}\to\mathbb{I}$ if it belongs to the  topological interior (in the space $\mathbb{I}$ with the natural topology) of at least one of the two sets:
\begin{equation*}
R^-(f,x)=\{y:f^{-1}(y)\cap(x-\varepsilon,x)\ne\emptyset\  \forall \varepsilon>0\},
\end{equation*}
\begin{equation*}
R^+(f,x)=\{y:f^{-1}(y)\cap (x,x+\varepsilon)\ne\emptyset\  \forall \varepsilon>0\}
\end{equation*}
Among the results of \cite{pawlak2009entropy} there is that, for Darboux, Baire-1 interval maps, the existence of at least one almost fixed point has strong dynamical consequences, as it implies infinite topological entropy as well as periodic points of every period; moreover, in the same assumptions, a fixed points exists in any open neighborhood of an almost fixed point. 
Notice that no continuous and even no regulated interval map can have almost fixed points, so that, in a sense, this concept distinguishes highly irregular interval dynamics. Here we limit ourselves to establish two very simple results linking the concept of almost fixed point to the framework of erasing substitutions.
\begin{lem}
If $\sigma$ is a $k-$block substitution, then the map $f_\sigma$ admits an almost fixed point only if $\sigma$ is erasing.
\end{lem}
\begin{proof}
If $\sigma$ is not erasing, then the map $f_\sigma$ admits right and left limits, say $x^r$ and $x^l$, at every $x\in\mathbb{I}$ (and they coincide outside $\mathbb{Q}_2$). In this case, the topological interior of $R^-(f_\sigma,x)$ and $R^+(f_\sigma,x)$ is empty.
\end{proof}
\begin{lem}
If $\sigma$ is an erasing $k$-block substitution verifying the optimality condition \ref{o.c}, then the map $f_\sigma$ admits at least one almost fixed point, namely $x_0=0.w_\epsilon^\infty$.
\end{lem}
\begin{proof}
For every $\varepsilon>0$, there is a positive integer $m$ so large that every point in $[w_\epsilon^m]$ is less than $\varepsilon$ apart from $x_0$. By the optimality condition, there is $v\in\{0,1\}^\omega$ such that $\sigma(v)=\sigma(w_\epsilon^m v)=w_\epsilon^\infty$. We can suppose that $v_1\dots v_k$ is strictly larger than $w_\epsilon$ in lexicographic order (the argument is analogous if it is strictly smaller), so that $x_0\in R^+(f_\sigma,x_0)$. Since the optimality condition also ensures that $[\sigma(v_1\dots v_k)]$ belongs to $R^+(f_\sigma,x_0)$, the point $x_0$ belongs to its topological interior. Notice that this also applies if $w_\epsilon=0^k$ or $1^k$, as the interior is meant with respect to the natural topology on $\mathbb{I}$.
\end{proof}
\noindent Notice that, by Remark \ref{fixedd}, when $\sigma$ is completely erasing, in any open neighborhood of $0.w_\epsilon^\infty$ there is a fixed point for $f_\sigma$.

\noindent It seems no coincidence that (some of) the properties guaranteed by the existence of an almost fixed point in the Darboux case are verified also for the maps $f_\sigma$ when $\sigma$ is completely erasing. This seems to suggest that the concept of almost fixed point has a dynamical significance also for Baire-1, non-Darboux interval maps. On the other hand, the fact that a further assumption is used herein on $\epsilon(\cdot)$ to get infinite entropy probably means that there is no straightforward strengthening of the results of  \cite{pawlak2009entropy}, but rather some suitable weakening of the Darboux property has to be invoked/introduced.
\vspace{0.3cm}

\begin{center}
\begin{tabular}{| p{4.5cm} p{0.3cm} p{7.5cm}|}
\hline
\vspace{13mm}\center{Erasing + OC} &\vspace{11mm} \center{$\implies$} & 
\begin{itemize}
\begin{footnotesize}
 \item $f_\sigma$ is continuous (non uniformly) on a set $\mathcal{C}^\sigma$ such that $\left\vert\mathbb{I}\setminus\mathcal{C}^\sigma\right\vert=\aleph_0$
 \item $f_\sigma^{-1}(x)$ is uncountable for every $x\in\mathbb{I}\setminus\mathcal{Q}_2^0$
 \item $f_\sigma$ is not bi-measurable
 \item $f_\sigma$ is Baire-1 and (in general) not Darboux
 \end{footnotesize}
   \end{itemize}\\
\hline
\vspace{12mm}\center{Strongly erasing + OC} & \vspace{10mm}\center{$\implies$} & 
\begin{itemize}
\begin{footnotesize}
\item The set of points with a dense $f_\sigma$-orbit is 

uncountable and dense in $\mathbb{I}$
\item $f_\sigma$ has 1/2-sensitive dependence on initial 

conditions
\item $f_\sigma$ is topologically mixing
\end{footnotesize}
\end{itemize}\\
\hline
\vspace{12mm}\center{Completely erasing + OC} & \vspace{10mm}\center{$\implies$} &
\begin{itemize}
\begin{footnotesize}
\item $f_\sigma$  exhibits Devaney chaos
\item $f_\sigma$ exhibits  Li-Yorke chaos
\item $f_\sigma$ has infinite topological entropy as soon as $\epsilon(\cdot)$ is asymptotically sublinear
\item $f_\sigma$ has almost fixed points
\end{footnotesize}
\end{itemize}\\
\hline
\vspace{3mm}\center{Boundedly erasing} & \vspace{1mm}\center{$\implies$} & 
\begin{itemize}
\begin{footnotesize}
\item $f_\sigma$ has trivial dynamics: there is $n$ such that 0 attracts in $n$ iterates every $x\in\mathbb{I}$ whose orbit does not intersect $\mathcal{Q}_2$
\end{footnotesize}
\end{itemize}\\
\hline
\end{tabular}
\end{center}
\vspace{0.15cm}
\textbf{Summary of the relations between erasing class of $\sigma$ and properties of $f_\sigma$.}
\vspace{0.3cm}

\section{Some further problems}
\label{furtherr}
\noindent Some possible generalizations of the questions addressed in this paper appear natural from either an analytical or a dynamical point of view. First of all, if we drop the assumption of uniqueness of the sequence $(h_i)$ made in Lemma \ref{limit_}, then the topological structure of the fibers becomes more intricate, because in general a fiber can be an uncountable union of sets each of which has a Cantor closure. It is not clear how this can affect the dynamics of the map $f_\sigma$. More generally, one can ask which are the minimal assumptions on an erasing substitution $\sigma$ to obtain the same dynamical properties proved here in the completely erasing case. 

\noindent Some natural questions also arise from the following simple argument: by construction, $f_\sigma(0.w)=f_\sigma(0.w_\epsilon w)$, so that the functional relation $f_\sigma(x)=f_\sigma\left(x2^{-k}+0.w_\epsilon\right)$ always holds. This induces a fractal structure in the graph of $f_\sigma$, because its restriction to $[w_\epsilon]$ is a horizontal compression plus a translation (the latter unless $w_\epsilon=0^k$) of the whole graph. Since a fractal structure arises, it seems natural to ask what is the link between the substitution $\sigma$ and the Hausdorff dimension of the graph of $f_\sigma$. We point out that the estimate of the Hausdorff dimension of the graph is not trivial even in the model case represented by the substitution $\sigma_3$ defined in Section \ref{prelim} (see \cite{dellacorte2021simplest}).




\bibliographystyle{elsarticle-num-names}
\bibliography{biblio_erasing_interval}




\end{document}